\documentclass[final]{l4dc2021}
\usepackage{times}
\usepackage{bbm}
\usepackage{amssymb}

\usepackage{algorithm}
\usepackage[noend]{algpseudocode}
\makeatletter
\def\BState{\State\hskip-\ALG@thistlm}

\RequirePackage{filecontents}
\makeatletter
\def\BState{\State\hskip-\ALG@thistlm}
\makeatother
\usepackage[noend]{algpseudocode}

\newcommand{\la}{\langle}
\newcommand{\ra}{\rangle}

\newcommand{\E}{\mathbb E}
\newcommand{\R}{\mathbb R}
\newcommand{\Ss}{\mathbb S}
\newcommand{\Var}{\text{Var}}

\newcommand{\de}{\delta}

\newcommand{\Prob}{\mathbb P}
\usepackage[utf8]{inputenc} 
\usepackage[T1]{fontenc}    
\usepackage{hyperref}       
\usepackage{url}            
\usepackage{booktabs}       
\usepackage{amsfonts}       
\usepackage{nicefrac}       
\usepackage{microtype}      
\usepackage{cleveref}
\usepackage{graphbox}
\usepackage{import}
\usepackage{multicol, blindtext}
\usepackage[font=small,labelfont=bf]{caption}

\usepackage{booktabs}

\usepackage{fullpage}
\usepackage{hyperref}
\usepackage{url}
\usepackage{multicol}

\usepackage{graphicx}

\newtheorem{fact}{Fact}
\newtheorem{assumption}{Assumption}

\newcommand{\be}{\begin{equation}}
\newcommand{\ee}{\end{equation}}

 \title[Log Barriers Safe Optimization]{Safe non-smooth black-box optimization with application to policy search}
 \author{%
 \Name{Ilnura Usmanova} \Email{ilnurau@control.ee.ethz.ch}\\
 \addr Automatic Control Laboratory,
 ETH Zürich, Switzerland
 \AND
 \Name{Andreas Krause} \Email{krausea@ethz.ch}\\
 \addr Machine Learning Institute,
 ETH Zürich, Switzerland%
 \AND
  \Name{Maryam Kamgarpour} \Email{mkamgar@control.ee.ethz.ch}\\
 \addr Automatic Control Laboratory,  ETH Zürich, Switzerland %
}

\begin{document}
	\maketitle
	\begin{abstract}
	  For safety-critical black-box optimization tasks, observations of the constraints and the objective are often noisy and available only for the feasible points. We propose an approach based on log barriers to find a local solution of a non-convex non-smooth black-box optimization problem $\min f^0(x)$ subject to  $f^i(x)\leq 0,~ i = 1,\ldots, m$, 
	  guaranteeing constraint satisfaction while learning an optimal solution with high probability. Our proposed algorithm exploits noisy observations to iteratively improve on an initial safe point until convergence. 
	  We derive the convergence rate and prove safety of our algorithm. We demonstrate its performance in an application to an iterative control design problem. \footnote{We thank the support of Swiss 
National Science Foundation, under the grant SNSF 200021\_172781, and ERC under the European Union's Horizon 2020 research and innovation programme 
grant agreement No 815943.}
\footnote{
Accepted by Learning for Dynamics \& Control conference (L4DC). The long version. }

	\end{abstract}

	\section{Introduction}
\paragraph{Motivation} Machine learning algorithms are increasingly being deployed for safety-critical emerging applications  
such as autonomous driving, personalized medicine and robotics. 
In such scenarios, safety and reliability of these algorithms is crucial. 
When the model is unknown, too complex or unreliable, it is common to adopt a black-box \emph{bandit} setup; our goal is to include safety in these learning techniques. 

\paragraph{Related work} 
In the optimization literature, several constrained optimization algorithms exist 
guaranteeing feasibility of the iterates given just local information about the constraints. These include 
 Feasible Sequential Quadratic Programming (FSQP) \citep{jian2005feasible,luo2012modifying,tang2014feasible}, the Method of Feasible Directions (MFD) \citep{zoutendijk1960methods}, and their variations. However, all these methods require first and/or second order information and do not consider the presence of noise. On the other hand, there are many works on derivative-free optimization, including non-convex and non-smooth problems \citep{balasubramanian2018zeroth, nesterov2017random,ghadimi2013stochastic,lan2013complexity}, based on finite difference gradient estimation techniques. However, these methods do not guarantee the feasibility of the points where measurements are taken with respect to unknown constraints. This issue can be addressed by interior point methods, where a barrier function is optimized. However, existing work on interior point methods typically require second order information. 

Safe learning with zero-th order (bandit) information has been considered in Bayesian Optimization  \citep{berkenkamp2016bayesian,sui2015safe} for non-convex constrained problems. 
As these works aim to compute a global optimum, they have to solve a nontrivial non-convex subproblem 
at each iteration. Moreover, for most common kernel functions, these algorithms require a number of measurements that is exponential on the dimensionality. This makes Safe Bayesian Optimization methods not always applicable to high dimensional problems. Moreover, appropriately choosing a prior distribution and the kernel parameters might not be a trivial task. 
Gradient based local methods usually do not suffer from these drawbacks. 

First order methods in application to barrier functions in the recent past were considered to have exponential runtime bounds due to the bad behavior of any barrier on the boundary of the feasible set. However, in the recent work \citep{hinder2019poly} the authors demonstrated that for smooth problems a gradient descent algorithm with adaptive step size on a log barrier function can be tractable, i.e., present attractive polynomial runtime convergence. 
Motivated by safe learning problems, the recent work \citep{usmanova2019log} extended this approach by \citep{hinder2019poly} to smooth non-convex optimization problems with zero-th order noisy measurements. However, the bound on the number of measurements was valid only in the case of \emph{a single} smooth constraint function. In this paper, generalizing the approach of  \citep{usmanova2019log}, we develop a safe algorithm for the \textit{non-smooth} non-convex constrained optimization problems subject to an arbitrary number of constraint functions. 
In Table \ref{table:1}, a comparison of our algorithm with existing methods for unconstrained and constrained zero-th order non-convex optimization is provided. 
In the first two algorithms a 2-point bandit feedback is assumed, i.e., it is possible to measure at several points with the same noise realization. 
 In our algorithm we assume a more realistic and more challenging setup with changing noise at each measurement.

Safe learning is widely used in control of unknown dynamical systems. For example, the work by \citep{dean2019safely} exploited system identification and robust optimal control to learn the safe linear quadratic regulator (LQR) subject to constraints on the state and input trajectories. 
There are many works aiming at guaraneeing safety while learning the optimal policy and dynamics in non-linear control, such as \citet{fisac2018general,berkenkamp2016bayesian,gillula2012guaranteed}. Often  Bayesian optimization approach is used to solve the above problems. However, Bayesian optimization might not enjoy acceptable scalability with dimensionality, thus limiting its applicability to control. 
Non-smoothness can also appear in some control problems such as bipedal walking, etc \citep{ames2014human}. In this paper, we consider an application of our method to safe learning for model-free control. We test our algorithm on a low dimensional control system, but theoretically the dependence on the dimensionality is only polynomial and the algorithm can be applied for higher dimensional problems. 

\begin{table*}[t]\label{T1}
    	\footnotesize
 		\begin{center}
 		\scalebox{0.9}{
 		\begin{tabular}[h!]{|p{13mm}|p{43mm}|p{35mm}|p{36mm}|p{35mm}|}
 				\hline
 				 Problem
 				& Unconstrained \& smooth objective
				& Known linear constraints \& smooth objective
				&  Single unknown smooth constraint \& smooth objective &  Several unknown non-smooth constraints \& non-smooth objective\\
                \hline
 				 Feedback
 				& 2-point bandit feedback
				& 2-point bandit feedback
				& 1-point bandit feedback & 1-point bandit feedback\\
                \hline
 				 Safety
 				& -
				& Yes \textit{(known constraints)}
				& Yes 
				& Yes\\
				\hline
				Optimality \newline condition &  Stationary point: \newline $\E\|\nabla f(x_T)\|_2 \leq \eta
 				$ &
 				$\eta$-stationary point: $\forall u\in D$  \newline $\E\la \nabla f(x_T), x_T - u\ra \leq \eta$ 
 				& $\eta$-approximate KKT point & $\eta$-approximate KKT point of the smoothed problem\\
                \hline
				Number of 
				measurements
				&  $O\left(\frac{d}{\eta^4}\right)$ or
				 $O\left(\frac{d}{\eta^{3.5}}
				\right) + \tilde O \left(\frac{d^4}{\eta^{2.5}}
				\right) $
				\newline \citep{balasubramanian2018zeroth}
				& $ O\left(\frac{d}{\eta^4
				}\right)$ 
				\newline \citep{balasubramanian2018zeroth}
				& $\tilde O\left(\frac{d^{3}}{\eta^{7}}
				\right)$ 
				 \newline \citep{usmanova2019log}
				& $\tilde O\left(\frac{d^3}{\eta^{9}}\right)$ \newline
				\textbf{(this work)}
				\\
				\hline
                \end{tabular}
	}
	\end{center}
			\caption{
			Upper bounds on number of  zero-th order oracle calls for non-convex smooth optimization algorithms.}
			\label{table:1}
\end{table*}

\paragraph{Our contributions} Our contribution is to propose an algorithm to find an approximate local solution to non-convex non-smooth cost functions subject to non-convex non-smooth constraints.  Furthermore, we prove the safety of the approach and derive its convergence rate 
in expectation in terms of the variance of the noise. 
Our algorithm is based on the log barrier gradient descent approach. 
Our convergence is with respect to an approximate stationary point of the smooth approximation of the problem. In the special case, where both the cost function and the constraints are smooth, we establish convergence to an approximate KKT point of the initial problem. We validate the  performance of our algorithm in application to a simple model-free control problem. 

\section{Problem statement}
\paragraph{Notations and definitions.} 
	Let $\|\cdot\|$ 
	denote the $l_2$-norm 
	on $\R^d$. A function $f: \R^d \rightarrow \R$ is called \textit{$L$-Lipschitz continuous} if
	$|f(x) - f(y)|\leq L\|x - y\|_2.$ 
	It is called \textit{$M$-smooth} if 
	the gradients $\nabla f(x)$ are $M$-Lipschitz continuous, i.e.,
	$\|\nabla f(x) - \nabla f(y)\|_2\leq M\|x - y\|_2. $ 
	A random variable $\xi$ is zero-mean $\sigma^2$-sub-Gaussian if  
	$\forall \lambda \in \R ~ 
	\E\left[ e^{\lambda \xi}\right] \leq \text{exp}\left(\frac{\lambda^2\sigma^2}{2}\right),$ 	which implies that  $\Var\left[\xi \right] \leq \sigma^2$ (this can be shown using Tailor expansion).  
	By $\mathbb S^d$ and $\mathbb B^d$ we denote the unit sphere and the unit ball in $\R^d$, respectively. We denote the characteristic function of a set $\mathcal X \subseteq \R^d$ by $\mathbb I_{\mathcal X}=\begin{cases}0,&x\in \mathcal X\\
	+\infty,&x\notin\mathcal X
	\end{cases}.$
	\paragraph{Problem formulation}
	We consider safe non-convex non-smooth constrained  optimization problem
	\hspace{-0.2cm}
	\begin{align}\label{problem}
	\min_{x\in \R^d} &~f^0(x) \nonumber\\
	\text{subject to }&~  f^i(x) \leq 0, ~~i =  1,\ldots,m,
	\end{align}
	where the objective function $f^0: \R^d \rightarrow \R$ and the constraints $f^i: \R^d \rightarrow \R$ are unknown $L$-Lipschitz continuous functions, and can only be accessed at feasible points $x$. We denote by $D$ the feasible set $D := \{x\in\R^d: f^i(x)\leq  0, i = 1,\ldots,m\}.$ 
	
	\begin{assumption}\label{A:1}
		The set $D$ has a non-empty interior, and there exists a known starting point $x_0$ for which $f^i(x_0)<0$  for $i = 1, \ldots, m.$
	\end{assumption}
	This assumption 
	is common in works on safe learning \citep{berkenkamp2016bayesian,sui2015safe} or on model-free LQR problems  \citep{fazel2018global, abbasi2019model}.
	\paragraph{Information.}
	We assume access to noisy measurements of all cost and constraint function values for any requested feasible point $x \in D$. In particular, the measurements are given by 
	$ F^i(x,\xi^i) = f^i(x) + \xi^i,~\forall i = 0,\ldots, m$  with 
	zero-mean sub-Gaussian noise $\xi^i$. The $\xi^i$'s are i.i.d. across different measurements. 
	
	\paragraph{Goal.} The goal of the algorithm is to find an approximate local optimum, using only noisy zeroth-order information. Moreover, it has to guarantee safety, i.e., constraint satisfaction with high probability for all points at which measurements are taken. For differentiable non-convex objective and constraints, the notion of local optimality is captured by KKT condition. In this setting, we show that our algorithm converges to an $\eta$-approximate KKT point for any $\eta > 0$ with constants $\tau_1,\tau_2 >  0$, that are fixed and independent on $\eta$:
	\begin{align}
	\vspace{-0.3cm}\label{KKT.1}
	&\lambda^i,-f^i(x) \geq 0,~ \forall i = 1,\ldots,m \tag{$\eta$-KKT.1}\\\label{KKT.2}
	&\lambda^i(-f^i(x)) \leq \tau_1\eta ,~ \forall i = 1,\ldots,m \tag{$\eta$-KKT.2}\\\label{KKT.3}
	&\|\nabla_x L(x, \lambda)\|_2 \leq \tau_2\eta, \tag{$\eta$-KKT.3}
	\end{align}
	where $L(x,\lambda) := f^0(x)+ \sum_{i = 1}^{m} \lambda^i f^i(x)$ is the corresponding Lagrangian function. For non-differentiable non-convex objective and constraints, local optimality conditions are less understood. In this case, we show convergence to an approximate KKT point of a corresponding smoothed problem. The smoothing will be described in the approach below. In Corollary \ref{sec:relation},  we connect the solution of the smoothed problem with   an approximate KKT point of the initial problem for the case of differentiable cost function and constraints.
	
\section{Proposed approach}
We propose to construct a log barrier for the smooth approximation of problem (\ref{problem}), and then apply the zero-th order stochastic gradient descent with an adaptive step size to minimize it. To estimate the gradient of the smoothed function we sample points around the current iterate, and take measurements at these points. A measurement is denoted as \emph{safe} if the point at which it is taken is feasible with high probability.
\subsection{Zero-th order gradient estimation.} \label{sec:oracle}
	Our algorithm uses a randomized zero-th order gradient estimator for cost and constraint functions. For a point $x_k$ the gradient is estimated taking $n_k$ samples uniformly at random on the unit sphere $\mathbb S^d$. 
	\begin{align}\label{G_estimator}
	G^i(x_k, \nu) := \sum_{j = 1}
	^{n_k}\frac{\hat G^i_j(x_k, \nu)}{n_k},~ \hat G^i_j(x_k, \nu) := d \frac{F^i(x_k+ \nu s_{kj}, \xi_{kj}^{i+}) - F^i(x_k, \xi_{kj}^{i-})}{\nu} s_{kj}  
	\end{align}
	for $i = 0,\ldots,m$, where all $\{[\xi_{kj}^{i+},\xi_{kj}^{i-}]\}_{j = 1\ldots,n_k}$ are i.i.d. sub-Gaussian random variables, $\nu > 0$ is the sampling radius, $s_{kj}$ are the sampled unit vectors. For the sampling radius $\nu\geq 0$ define the $\nu$-\textit{smoothed} estimate of the function $f(x)$ by
	$f_{\nu}(x):= \E_{b}f(x + \nu b),$ where  $b$ is uniformly distributed on the unit ball $\mathbb B^d$ 
	Then  $\E_{s_k, \xi_k}  G^i(x_k,\nu) = \nabla f^i_{\nu}(x_k)$ \citep{flaxman2005online, nesterov2017random, balasubramanian2018zeroth}.  This shows that in expectation we can get a gradient of the smooth estimate $f_{\nu}(x)$ of non-smooth function $f(x)$ using randomized gradient estimator. The following properties hold:
	\begin{enumerate}
		\item[1)] The gradient $\nabla f_{\nu}(x)$ is Lipschitz continuous with constant $M_{\nu}$ satisfying 
		$M_{\nu}\leq \frac{\sqrt{d}L}{\nu}$. (For the proof of this property see Appendix \ref{proof:A00}.)
		\item[2)] 
		$\forall x \in \R^d ~  |f_{\nu}(x) - f(x)|\leq \nu L$ \citep{hazan2016graduated}.
		\item[3)]The function $f_{\nu}(x)$ is Lipschitz continuous with $L_{\nu}\leq L,$ which implies $\|\nabla f_{\nu}(x)\|\leq L$ for differentiable  $f_{\nu}(x)$. (For the proof of this property see Appendix \ref{proof:A0}.)
	\end{enumerate}

\subsection{ Smoothed log barrier function.}
We  address the safe learning problem using the log barriers approach. 
 Define $f^c(x)= \max_{i = 1, \ldots, m} f^i(x),$ which is in general non-smooth and non-convex. The logarithmic barrier with parameter $\eta>0$ of the initial problem with the constraints replaced with $f^c(x)$ above is defined as 
$ B_{\eta}(x) = f^0(x) - \eta \log (-f^c(x)).$ 
We define the \textit{locally} \textit{smooth} barrier function and its gradient using \textit{smoothed} functions $f^0_{\nu}(x)$ and $f^c_{\nu}(x)$:
\begin{align}
B_{\eta,\nu}(x)& = f^0_{\nu}(x) - \eta \log (-f^c_{\nu}(x)),\\
\nabla B_{\eta,\nu}(x) &= \nabla f^0_{\nu}(x) + \eta \frac{ \nabla f^c_{\nu}(x)}{-f^c_{\nu}(x)},
\end{align}
 It is evident that the gradient of the barrier grows to infinity while converging towards the boundary 
 and hence, the barrier function cannot be smooth. \textit{Local} smoothness of barrier function refers to existence of a value $M_{2,\nu}(x)$ that bounds the change in barrier gradient $\frac{\|\nabla B_{\nu,\eta}(x) - \nabla B_{\nu,\eta}(y)\|}{\|x-y\|}$ for a ball around point $x$, where $M_{2,\nu}(x)$ is determined later in Appendix \ref{proof:F} in (\ref{l2:bound}).
Our goal is to design an algorithm that converges to a locally optimal  
point $x^*$ of the smoothed log barrier $ 
B_{\eta,\nu}(x)$, which is basically an unconstrained approximation of constrained smoothed problem:
$ f^0_{\nu}(x) + \mathbb I_{f^c_{\nu}(x)\leq 0}.$ Then, we show that this point $x^*$ satisfies  $\eta$-approximate KKT conditions for the smoothed problem $\min_{f^c_{\nu}(x)\leq 0} f^0_{\nu}(x)$.


\subsection{Log barrier gradient estimator.}\label{sec:barrier_estimator}
First, we need to propose a way to estimate 
$\nabla B_{\eta,\nu}(x_k) = \nabla f^0_{\nu}(x_k) + \eta \frac{ \nabla f^c_{\nu}(x_k)}{-f^c_{\nu}(x)}.$ 
To estimate $\nabla f^0_{\nu}(x_k)$  and $\nabla f^c_{\nu}(x_k)$ we can use $G^0(x_k,\nu)$ defined by (\ref{G_estimator}), and $G^c(x_k,\nu)$ defined similarly but with $F^c(x_k,\xi_k) = \max_i F^i(x_k,\xi_k^i)$ instead of $F^i(x_k,\xi_k^i)$. 
However, to estimate the denominator we propose to use a lower confidence bound on $-f^c_{\nu}(x_k)$, constructed as follows. 
Given $\de>0$ for $i=1,\ldots,m$ define $\hat F^i(x_k):= \frac{\sum_{j = 1}^{n_k} F^i(x_k,\xi^{i-}_{jk})}{n_k} + \frac{\sigma}{\sqrt{n_k}} \sqrt{\ln \frac{1}{\de}}.$  We show in Appendix \ref{proof:A} that
$\Prob\{f^i(x_k) \leq \hat F^i(x_k) \}\geq 1-\de.$ 
We define by $\hat F^c_{\nu}(x_k) := \max_{i=1,\ldots,m}\hat F^i(x_k) + \nu L$ an upper confidence bound on $f^c_{\nu}(x_k)$ and by $\hat \alpha_k:= |\hat F^c_{\nu}(x_k)|$ a lower confidence bound on both $|f^c(x_k)|$ and $|f(x_k)|.$ More precisely,  
$\Prob\left\{\hat \alpha_k \leq \min\{|f^c_{\nu}(x_k)|, |f^c(x_k)|\}\right\}\geq 1-\de.$ The proofs of the above statements are shown in Appendix \ref{proof:A}. 
Then, we propose to estimate $\nabla B_{\eta,\nu}(x_k)$ by 
\begin{align}\label{g_estimator}
g_k :=  G^0(x_k, \nu)  + \eta\frac{G^c(x_k, \nu)}{\hat \alpha_k}.
\end{align}
Later in Fact \ref{lemma:2}, we bound the deviation $\|g_k - \nabla B_{\eta,\nu}(x_k)\|$ with high probability.   Next, to define our algorithm we need to make a second assumption.
\begin{assumption}\label{assumption:2}
	Let $D'\subseteq D$ be the subset defined by
	$D'=\left\{x\in \R^d: f^c_{\nu}(x) + \eta 
	\leq 0\right\}.$  There exists $l>0$ such that the norm of the gradient 
	$\nabla f^c_{\nu}(x)$ is lower bounded on 
	$D\setminus D'$ by $l$, i.e, $0 < l \leq \|\nabla f^c_{\nu}(x)\|\leq L.$
	\end{assumption}
	Assumption 2 is needed to demonstrate that close to the boundary of the constraint set the term in the barrier gradient related to constraints becomes large enough to push the step direction away from the boundary back to the feasible set. 
	A slightly modified Mangasarian-Fromovitz Constraint Qualification (MFCQ) that holds for all points leads to the satisfaction of Assumption \ref{assumption:2}. We show this in Appendix \ref{proof:assumption}. 
	
The proposed stochastic zero-th order algorithm is defined in Algorithm \ref{safe-barrier-0} below:
\begin{algorithm}[h]
	\caption{Stochastic Zero-th Order  Logarithmic Barrier Method \textbf{(ZeLoBa)}} \label{safe-barrier-0}
	\small
	\begin{algorithmic}[1]
		\BState \emph{Input:} $x_0 \in D$, number of iterations $K$, $\eta > 0$
		, $L,l > 0$, $C =\frac{l^2}{8L^2}$, $\nu = \frac{C\eta}{L}$, $\{n_k\}_{k = 1}^K$ defined in Lemma \ref{lemma:3}
		\State\While {$k \leq K$}{
		\State Sample $n_k$ vectors $s_{kj}, j = 1,\ldots,n_k$ independently from the uniform distribution on $\mathbb S^d$;
		\State Take $n_k$ noisy measurements of each function $f^i(x),~ i = 0,\ldots,m$ at points $F^i(x_k, \xi^{i-}_{kj}), F^i(x_k + \nu s_{kj}, \xi^{i+}_{kj})$;
		\State Compute an estimator $g_k$ of  $\nabla B_{\eta,\nu}(x_k)$ using (\ref{g_estimator});
		\State Compute $\gamma_k = \frac{1}{\|g_k\|}
		\min\left\{\frac{\hat \alpha_k}{2Lk^{2/5}},\frac{1}{k^{3/5}}\right\}
		$;
		\State $x_{k+1} = x_{k} - \gamma_k g_k$;
	}
         \State Sample $R$ from a discrete random distribution $\Prob\{ R  = k\}=  \frac{\gamma_k\|g_k\|}{\sum_{k=1}^K \gamma_k \|g_k\|}$
         \State $\lambda_{R} =\frac{\eta}{\hat\alpha_{R}}$
	    \BState \emph{Output:} $x_R, \lambda_R$
	\end{algorithmic}
\end{algorithm}\\
In the above, $l$ is the constant defined in Assumption \ref{assumption:2}.  
Our algorithm is defined for fixed $\eta$. In practice interior point methods often use decreasing $\eta$. Our algorithm can be used for inner iterations of the classical log barrier method with decreasing $\eta$. 

\section{Safety and convergence analysis}
From the algorithm we require the safety of the iterates, $f^c(x_k) \leq 0$, and the safety of the measurements, $f^c(x_k+ \nu s_{kj}) \leq 0$, with high probability. 
Also, we require convergence to a stationary point of the smoothed function in expectation. Here, we show that these properties hold for ZeLoBa algorithm.

\subsection{Safety.}\label{sec:safety}
Given the required accuracy $\eta$, the smoothing parameter $\nu$ (which is also the sampling radius) has 
to be conservative enough to guarantee constraint satisfaction  at any measured point $x_k + \nu s_{kj}$ of ZeLoBa algorithm. 
 Thus, we need to show that the iterates $x_k$ always keep a sufficient distance $\Lambda>0$ from the boundary, namely $-f^c_{\nu}(x_k)\geq\hat\alpha_k \geq \Lambda.$  
To show the above, we first need to bound the deviation of $g_k$ from $\nabla B_{\eta,\nu}(x_k)$. 
Define the deviation by $\zeta_k := g_k - \nabla B_{\eta,\nu}(x_k).$  The deviation $\zeta_k$ is dependent on deviations $\Delta^0_k := G^0(x_k,\nu) - \nabla f^0_{\nu}(x_k),\Delta^c_k := G^c(x_k,\nu) - \nabla f^c_{\nu}(x_k)$, thus, we bound these latter terms first.  
We denote $\Sigma:= (d+1)\sqrt{\ln\frac{1}{\de} + \ln (2K+1)}\left(\sqrt{2}\sigma + L \nu\right).$  From the sub-Gaussian property of the noise $\xi^{i\pm}_{kj}$ and 
$L$-Lipschitz continuity of $f^i(x), i = 0,\ldots,m$, we have:
\begin{fact}\label{lemma:0}
 For deviations $\Delta^j_k = G^j(x_k,\nu) - \nabla f^j_{\nu}(x_k), j =\{0,c\}$, 
 we have 
$\E\|\Delta^j_k\|^2 
\leq 
 \frac{d^2}{n_k}\left(L^2 + \frac{2\sigma^2}{\nu^2}
\right).$\\
For all points $x_k$ with $k \leq K$ 
  we have 
$\Prob\bigg\{\forall k =1,\ldots,K ~ \|\Delta_k^j\|\leq \frac{\Sigma}{\nu\sqrt{n_k}}\bigg\} \geq 1-\de.$
\end{fact}
For the proof see Appendix \ref{proof:B}. Using this result, we can get the following bound on $\zeta_k$:
\begin{fact}\label{lemma:2}
For deviation $\zeta_k = g_k - \nabla B_{\eta,\nu}(x_k)$, we have $\E\|\zeta_k\| \leq  \frac{(d+1)(\sqrt{2}\sigma+L\nu)}{\nu\sqrt{n_k}}\left(1+\frac{2 \eta}{\hat\alpha_k}\right).$\\
For all $k\leq K$ we have
$\Prob\left\{\forall k =1,\ldots,K ~\|\zeta_k\| \leq  \frac{\Sigma}{\nu\sqrt{n_k}}\left(1+\frac{2 \eta}{\hat\alpha_k}\right)\right\}\geq 1-\de.$
\end{fact}
For the proof see Appendix \ref{proof:C}. 
From the above facts, observe that if we keep the iterates $x_k$ away from the boundary, $\hat\alpha_k \geq \Lambda > 0$, we can bound  the deviation $\zeta_k$. Luckily, the Log Barrier gradient  approach with sufficiently large number of measurements in ZeLoBa ensures this property, 
as shown in the following lemma.

\begin{lemma}\label{lemma:3}
Under Assumption \ref{assumption:2}, 
if the initial point satisfies $- f^c_{\nu}(x_0) \geq 2C\eta$ with $C=\frac{l^2}{8L^2}$, then for all iterates $x_k$ of ZeLoBa algorithm with $n_k  \geq \frac{4\Sigma^2 (C+1)^2}{\nu^2C^2L^2}$  we have 
$\Prob\{ \hat\alpha_k \geq C \eta ~\forall k \leq K\}\geq 1-\de.$  
	\end{lemma}
\paragraph{Proof sketch:}
The idea is to show that the satisfaction of $\hat \alpha_k\geq C \eta$ and $-f^c_{\nu}(x_k)\geq 2C\eta$ for iteration $k$ implies the same bounds for the next iteration $k+1$ with high probability. To prove this,
we divide the condition $-f^c_{\nu}(x_k)\geq 2C\eta$ into two following cases. 
\textbf{Case 1.} $-f^c_{\nu}(x_k)\geq 4C\eta$, i.e., $x_k$ is far from the boundary of the constraint set. Then, in the next iteration $-f^c_{\nu}(x)$ cannot decrease more than twice due to the choice of the step size and $L$-Lipschitz continuity of $-f^c_{\nu}(x)$. Thus, for $x_{k+1}$ the bound $-f^c_{\nu}(x_{k+1})\geq 2C\eta$ holds.   
\textbf{Case 2.}  $-f^c_{\nu}(x_k) \leq 4C\eta$, i.e., $x_k$ is close enough to the boundary. In this case, we show that $-g_k$ pushes $x_{k+1}$ away from the boundary. That is, $-g_k$ is the descent direction for $f^c_{\nu}$: $\la g_k, \nabla f^c_{\nu}(x_k)\ra\geq 0.$ 
This is because $g_k$ defined in (\ref{g_estimator}) can be expressed as a sum of $\frac{\eta}{\hat \alpha_k}\nabla f^c_{\nu}(x_k) $ and $\nabla f^0_{\nu}(x_k) + \zeta_k$, and the first term will be dominating. Indeed, close to the boundary the factor $\frac{\eta}{\hat\alpha_k}$ is large and  $\|\nabla f^c_{\nu}(x_k)\|$ is lower bounded $\|\nabla f^c_{\nu}(x_k)\| \geq l > 0$ due to Assumption \ref{assumption:2}. 
Moreover, the step size $\gamma_k$ is sufficiently small to guarantee that $f^c_{\nu}(x_{k+1})$ will not increase compared to $f^c_{\nu}(x_k)$ due to the $M_{\nu}$-smoothness. Consequently,  $ 
	-f^c_{\nu}(x_{k+1}) \geq 2C\eta
	$ holds for both cases. 
	This  implies  $
	\hat\alpha_{k+1} \geq C\eta
	$ with high probability. Everything above holds conditioned on the previous iteration $k$. Carefully combining the conditional probabilities along $k = 1,\ldots,K$, we get the result of the lemma. The full proof can be found in Appendix \ref{proof:D}.
	
	The above lemma implies that the sampling radius $\nu = \frac{C\eta}{L} \leq \frac{\hat \alpha_k}{L}$ is safe. Hence, our algorithm is safe:
\begin{proposition}\label{proposition} Let Assumptions 1,2 hold and $n_k \geq  \frac{4\Sigma^2 (C+1)^2}{\nu^2C^2L^2}.$  Then all iterations $x_k$ and measurement points $x_k + \nu s_{kj}$ generated by ZeLoBa algorithm are safe, 
namely, 
$ \Prob\{f^c(x_{k}) \leq 0 ~ \forall k\leq K \} \geq 1 - \de$ 
and $\Prob\{f^c (x_{k}+\nu s_{k,j}) \leq 0~ \forall k\leq K \}\geq 1 - \de.$ 
\end{proposition}
For the proof see Appendix \ref{proof:E}. 

\subsection{Convergence.}
\begin{theorem}\label{theorem:4}
Under Assumptions 1,2, for $n_k  \geq  \frac{4\Sigma^2 (1+1/C)^2}{C^2\eta^4}$ and $K = \tilde O\left(\frac{d^{5/6}}{\eta^5}\right)$ iterations of ZeLoBa algorithm we have
 	$\E \|\nabla B_{\eta,\nu}(x_R)\|
	\leq 5\eta.$
	This implies that for the pair $(x_R, \lambda_R)$ in expectation $\eta$-approximate KKT condition holds:
	\begin{align}
	\vspace{-0.3cm}
	&\Prob\{\lambda_R,-f^c_{\nu}(x_R) \geq 0\}\geq 1-\de, \tag{$\eta$-KKT.1}\\
	&\Prob\{\lambda_R(-f^c_{\nu}(x_R)) \leq 3\eta\}\geq 1-\de \tag{$\eta$-KKT.2} ,\\
	&\E \|\nabla_x L(x_R, \lambda_R)\|_2 \leq 5\eta. \tag{$\eta$-KKT.3}
	\end{align} The total number of measurements required is $N_K = n_k\cdot K = O(\frac{d^3}{\eta^9}).$
\end{theorem}
\paragraph{Proof sketch:} The proof is based on standard non-convex analysis techniques. The log barrier $B_{\eta,\nu}(x)$ is only locally smooth with smoothness parameter $M_{2,\nu}(x_k)\leq 
 M_{\nu}\left(1+\frac{2\eta}{\hat \alpha_k}\right) + \frac{4L^2\eta}{\hat \alpha_k^2}$ for all the points within the ball with radius $\gamma_k$ around $x_k.$ Using the local smoothness, we bound the improvement in barrier value $B_{\eta,\nu}(x_{k+1}) - B_{\eta,\nu}(x_{k})$ for each iteration $k$. Summing this together for all $k\leq K$ provides the bound on 
 $\sum_{k = 1}^K \frac{\gamma_k}{\|g_k\|}\|\nabla B_{\eta,\nu}(x_{k})\|.$ 
This expression represents scaled $\E_R\|\nabla B_{\eta,\nu}(x_{R})\|$ for $R$ defined at Step 8 of ZeLoBa algorithm.  That is, we get $\E \|\nabla B_{\eta,\nu}(x_R)\| \leq 5\eta$ for $$K \geq \max\left\{\left(\frac{LD_f}{C\eta^2}\right)^{5/2}, \left(\frac{L^2\sqrt{d}(1+1/C)}{C^2\eta^3}\right)^{5/3}, \left(\frac{5\ln 1/\eta}{C\eta}\right)^5\right\} = \tilde O\left(\frac{d^{5/6}}{\eta^5}\right).$$ By construction, $\nabla B_{\eta,\nu}(x_R)$ equals to $ \nabla_x L(x_R,\lambda_R)$ for the smoothed problem, that implies ($\eta$-KKT.3). ($\eta$-KKT.1) follows from Proposition \ref{proposition}.  ($\eta$-KKT.2) follows from $\lambda_R = \frac{\eta}{\hat\alpha_R}$ and Lemma \ref{lemma:3}. 
The full proof is 
 in Appendix \ref{proof:F}.\\ 
\textbf{Remark:} The obtained bound on the number of measurements, $O(\frac{d^3}{\eta^9})$, is $\frac{1}{\eta^2}$ times worse compared to \cite{usmanova2019log}. This comes as a price for non-smoothness.
This difference agrees 
with the difference $\frac{1}{\eta^2}$ in upper bounds in other works on zero-th order optimization such as \cite{duchi2015optimal}.
\begin{corollary}\label{sec:relation}
 If the initial objective and constraints are differentiable, then the result obtained in Theorem \ref{theorem:4} 
 entails satisfaction of the approximate KKT condition for the initial problem (\ref{problem}). 
\end{corollary}
\begin{proof}
 We define $\hat \lambda_R \in \R^m $, where
$
\hat \lambda^i_R = \begin{cases}
0 
&, i \notin \arg\max_i \hat F^i(x_R),\\
\frac{\eta}{-\hat F^c_{\nu}(x_R)} 
&, i \in \arg\max_i \hat F^i(x_R) 
\end{cases}.
$
We can easily see that condition  \ref{KKT.1}  holds with high probability by construction: $-f^c_{\nu}(x_R)\geq \hat\alpha_R\geq 0.$ 
Condition  \ref{KKT.2} holds for all $i \notin \arg\max_i F^i(x_R)$ since $\hat\lambda^i_R$ is just equal to $0$.
For $i \in \arg\max_i \hat F^i(x_R)$, we have 
$\frac{\eta}{-\hat F^i(x_R)}(-f^i(x_R)) \leq \eta + \eta \frac{\hat F^i(x_R) - f^i(x_R)}{-\hat F^i(x_R)} \leq \eta+ \eta \frac{\sigma\sqrt{\ln 1/\de}/\sqrt{n_k}}{\hat\alpha_R} \leq 3\eta.$ 
Finally, we can verify that condition  \ref{KKT.3} holds as follows, using $\|\nabla f^i_{\nu}(x) - \nabla f^i(x)\| \leq \nu L d$ \citep{nesterov2017random}:
\begin{align*}
&\E\|L(x_R,\hat\lambda_R)\| = \E\|\nabla f^0(x_R) + \sum_{i = 1}^m\hat\lambda^i \nabla f^i(x_R)\|= \E\left\|\nabla f^0(x_R) + \frac{\eta \nabla f^i(x_R)}{-\hat F^c_{\nu}(x_R)}\right\|  \\
& \leq \E\bigg\|\nabla f^0_{\nu}(x_R)  + \frac{\eta \nabla f^i_{\nu}(x_R)}{-\hat F^c_{\nu}(x_R)}\bigg\| + \left\| \nabla f^0(x_R) -  \nabla f^0_{\nu}(x_R)\right\| +  \eta\left\|\frac{ \nabla f^i(x_R)}{-\hat F^c_{\nu}(x_R)} - \frac{ \nabla f^i_{\nu}(x_R)}{-\hat F^c_{\nu}(x_R)}\right\| \\
&\leq  \|\nabla B_{\eta,\nu}(x_R)\| + \eta\left(1+\frac{\nu L d}{\hat\alpha_R}\right)\leq \|\nabla B_{\eta,\nu}(x_R)\|+\eta (d+1) \leq \eta(6 +d).\end{align*} 
\end{proof}
In case of non-smooth objective and constraints, we do not know yet how to relate directly the result of Theorem \ref{theorem:4} to the original problem (\ref{problem}). This is a direction of the future research. 
\section{Experiments}
We consider the application to safe iterative controller design. 
Consider the basic unicycle dynamics 
    $ \dot x = v \cos \theta, 
     \dot y = v \sin \theta,
     \dot \theta = \omega.$ 
    Here the states $q = [x,y,\theta]$ describe the spatial coordinates $x,y$ and the direction angle $\theta$. The control inputs $u = [v, \omega]$ describe the speed and the angular velocity. Since the simple Euler discretization is valid only when the sampling period $dt$ is sufficiently  small, we use a discretized model of the unicycle based on direct integration of the dynamics 
    \citep{nino2006discrete, adinandra2012practical}: 
\begin{align*}
q_{t+1} = \begin{bmatrix}
 x_{t+1}\\
 y_{t+1}\\
 \theta_{t+1}
\end{bmatrix} = 
 q_{t} + \begin{bmatrix}
2v_t + \gamma(\omega_t)\cos (\theta_t + \frac{dt}{2}\omega_t)  \\
2v_t + \gamma(\omega_t)\sin (\theta_t + \frac{dt}{2}\omega_t)\\
dt\omega_t
\end{bmatrix}, \; \gamma(\omega_t) =\begin{cases}
\sin (\frac{dt}{2}\omega_t),&\omega_t \neq 0\\
\frac{dt}{2},&\omega_t = 0
\end{cases}.
\end{align*}
We choose a memoryless linear  feedback law $u_{t+1} = U q_{t}$, where $U \in \R^{3\times 2}$  is the optimizing parameter. The state sequence determined by $U$ is denoted by $q_t(U)$, $t=1, \dots,T$ where $T$ is the planning horizon. The goal is to lead the vehicle from a starting point $q_A$ to a goal destination $q_B$ while avoiding collision with high-probability. The cost function is defined as
    $\sum_{t = 1}^{T} \|q_t(U) - q_B\|^2.$ 
The constraints are formulated such that the trajectory does not collide with the the ball shaped obstacle placed at $(x_C, y_C)^T$ with radius $1$. The resulting constrained optimization problem is as follows:
 \begin{align*}
\min_{U\in \R^{3\times 2}}  &~\frac{1}{T}\sum_{t = 1}^{T} \|q_t(U) - q_B\|^2 \\
    \text{subject to }&~ 1 - \left\|(x_t(U),y_t(U))^T - (x_C,y_C)^T\right\|^2 \leq 0,~t = 1,\ldots, T.
 \end{align*}
In the zero-th order oracle approach, we assume no knowledge of the dynamics, the constraints or the cost functions. We only assume noisy measurements of the cost function and the constraints. 
Thus, we address this problem using the ZeLoBa algorithm. We set the parameters of the algorithm to $\nu_k= \min\{\frac{\eta}{L},\frac{\hat\alpha_k}{L}\}$ for safety, $L = 40$ set by trial, $n_k = 7$, $K = 500$, and initialize the algorithm with a safe control policy. The algorithm iteratively improves the controller while avoiding the constraints. The total number of measurements is $N_K = 3500.$ 
In Figure \ref{F} a) below we demonstrate the achieved results of 20 trials of the stochastic ZeLoBa algorithm with the fixed initialization. 
In none of the trials the constraints were violated. 
In Figure \ref{F} b) we show the trajectory generated by $U_0$ controller. 
In Figure \ref{F} c) we demonstrate an example of the trajectory generated by the final controller obtained during one of the trials of stochastic ZeLoBa algorithm. 
\begin{figure*}[t!]
    \centering
    \includegraphics[height=2in]{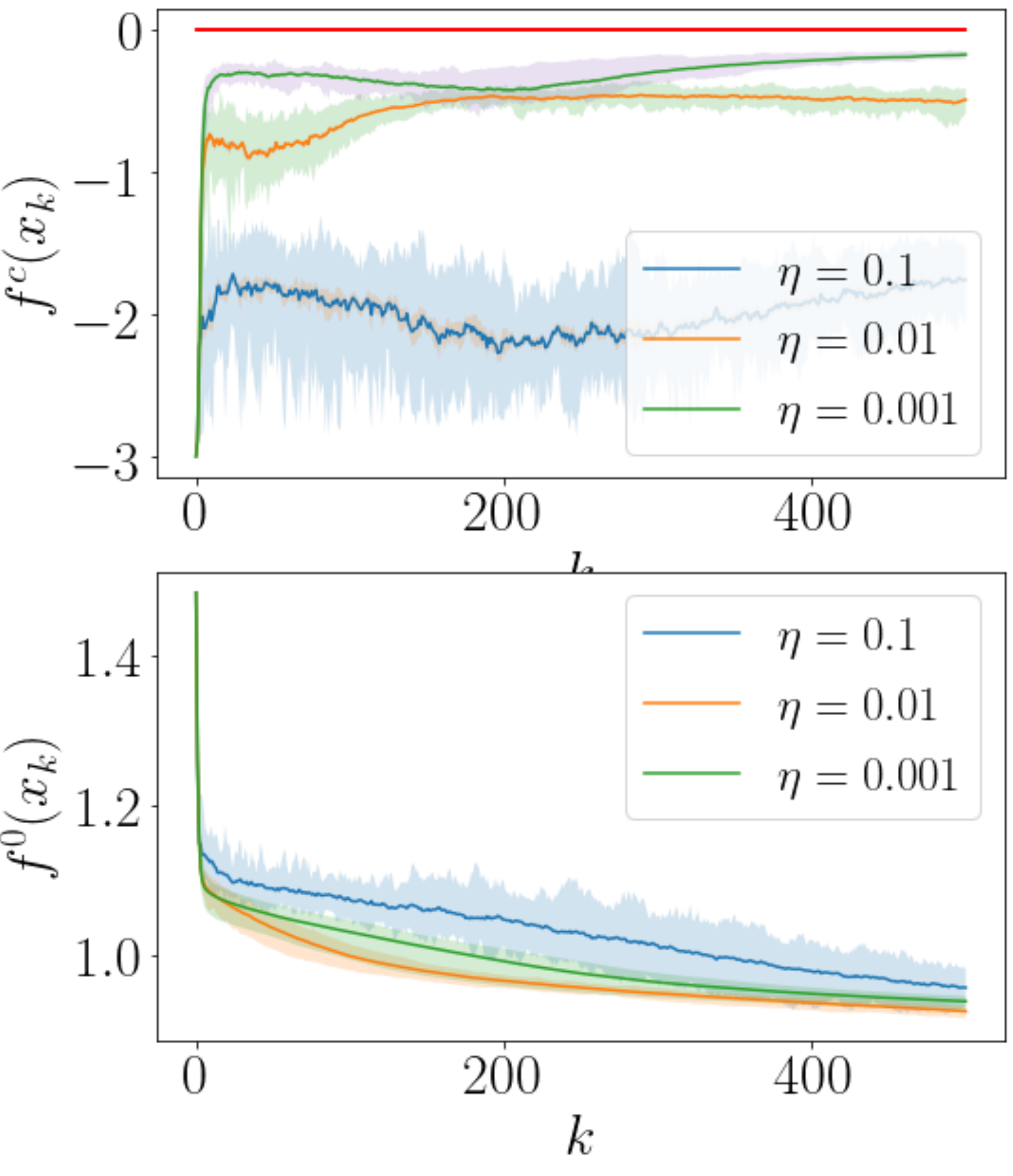}
    \includegraphics[height=2in]{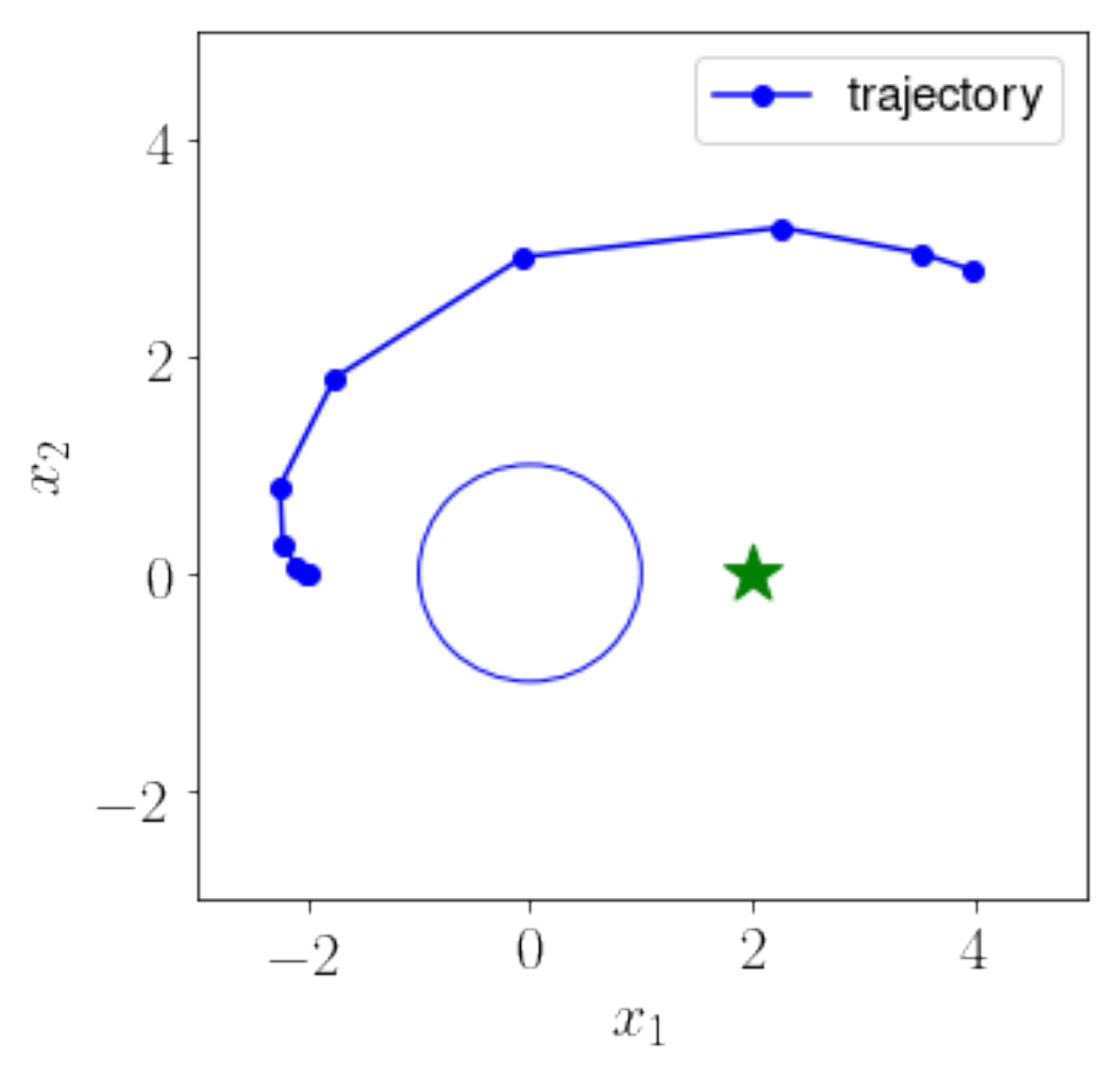}
    \includegraphics[height=2in]{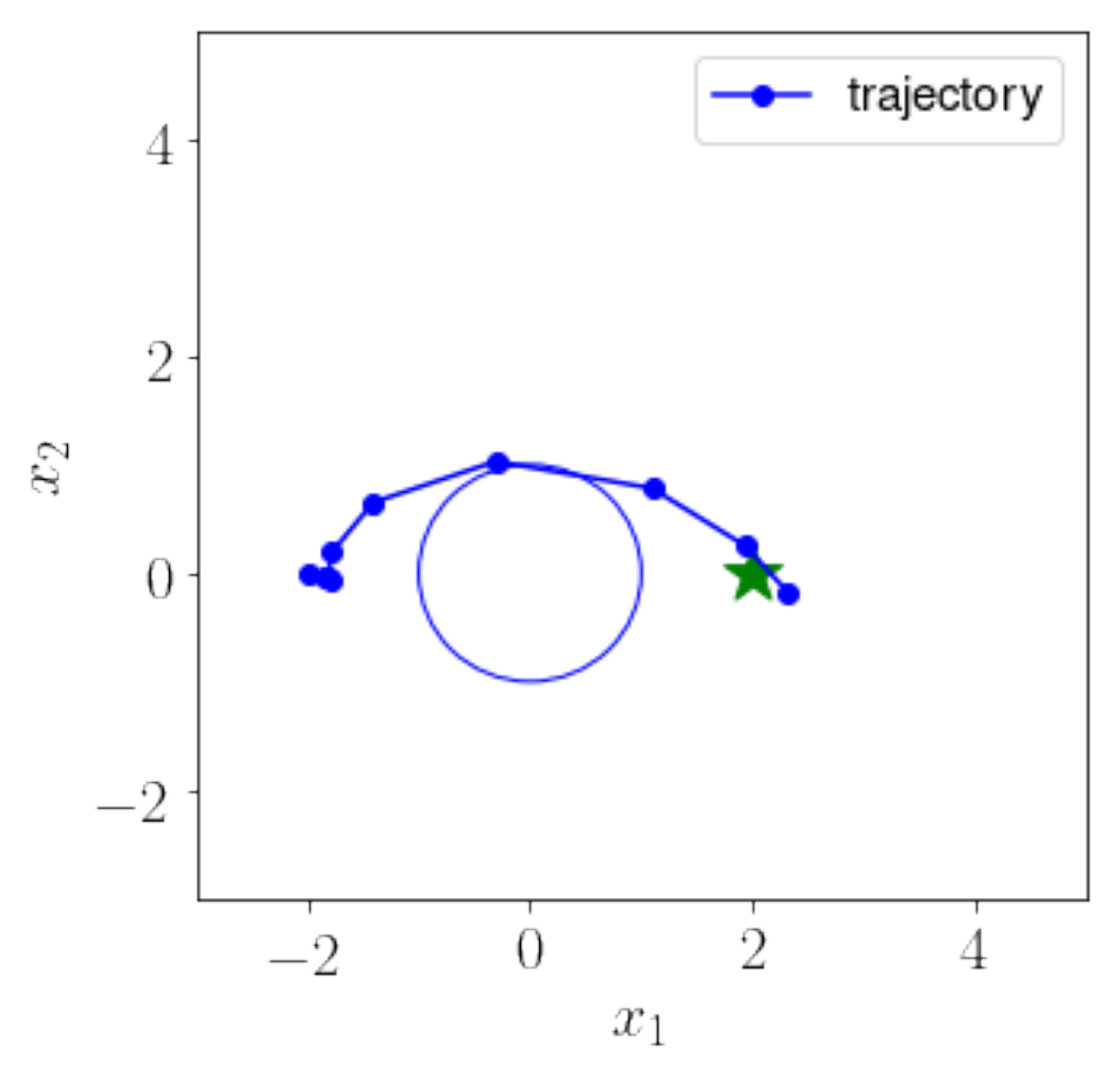} 
     \vspace{-0.4cm}
    \caption{a) \textit{Left.} Maximum constraint value (\textit{top}) and the objective value (\textit{bottom}) for 20 experiments; 
    b) \textit{Middle.} Control trajectory with $U_0$; c) \textit{Right.} Final control trajectory obtained by ZeLoBa with $\eta = 0.001.$}
    \label{F}
\end{figure*}
\addcontentsline{toc}{section}{Bibliography}
\bibliography{bibliography}
\clearpage
\appendix
\section{Proof of property 1): $M_{\nu}$ smoothness of $f_{\nu}$, Section \ref{sec:oracle}:}\label{proof:A00}
By definition and the property of the smoothed function $\nabla f_{\nu}(x) = \E_s d \frac{f(x+\nu s) - f(x)}{\nu}s,$ where $s\sim \mathcal U(\Ss^d)$. 
Hence $$\nabla f_{\nu}(x) - \nabla f_{\nu}(y)  = \E_s\left[ d \frac{f(x+\nu s) - f(x)}{\nu}s -  d \frac{f(y+\nu s) - f(y)}{\nu}s\right] =  d \E_s \frac{f(x+\nu s) - f(y+\nu s)}{\nu}s.$$
Let us denote by $\delta_f(s)$ the function: $$\delta_f(s) := f(x+\nu s) - f(y+\nu s).$$
Then, we have:
\begin{align}\label{eq:A10}
    \|\nabla f_{\nu}(x) - \nabla f_{\nu}(y)\|_2 =
\frac{d}{\nu} \|\E_s \de_f(s) s\|_2.
\end{align}
First, note that the absolute value of $\delta f(s)$ is bounded by
$$|\delta_f(s)| = |f(x+\nu s) - f(y+\nu s)|\leq L\|x-y\|_2.$$
Assume that $r\in\R^d :\|r\|_2 = 1$ is the unit vector of the direction of $\E_s \de f(s) s.$ Then,
\begin{align*}
\|\E_s \de f(s) s\|_2 
& = \la \E_s \de_f(s) s, r\ra = \E_s  \de_f(s)  \la s, r\ra
\\
& =  \frac{1}{2}\E_s  [\de_f(s)  \la s, r\ra|\la s, r\ra \geq 0]  + \frac{1}{2}\E_s  [\de_f(s)  \la s, r\ra|\la s, r\ra<0]
\\
& = \frac{1}{2}\E_{s\in{\Ss^d}, \la s, r\ra\geq 0}  [\de_f(s)  \la s, r\ra]  + \frac{1}{2}\E_{s\in{\Ss^d}, \la s, r\ra\geq 0}  [\de_f(-s)  \la s, r\ra].
\end{align*}
Note that in the above terms the multiplicands $\la s,r\ra$ are positive. 
Therefore, we can bound the whole product $\de_f(s)\la s,r\ra\leq |\de_f(s)|\la s,r\ra \leq L\|x-y\|_2 \la s,r\ra.$ 
Then, this has to be integrated over the half-sphere $s\in S^d, \la s,r\ra\geq 0,$ which we denote by $\Ss^d_+.$ 
Consequently, using (\ref{eq:A10}) we get 
\begin{align}\label{eq:A11}
    \|\nabla f_{\nu}(x) - \nabla f_{\nu}(y)\|_2 
 =\frac{d}{\nu}  \E_s \de_f(s) \la s, r\ra \leq \frac{d}{\nu}L\|x-y\|_2  \E_{s\in \Ss^d_+} \la  s, r\ra.
\end{align} Note that the expectation over the half sphere $s\sim \mathcal{U}(\Ss^d_+)$ of projection of $s$ onto the one direction $r$ is \begin{align}\label{eq:A1}
 \E_{s\in\Ss^d_+}\la s, r\ra 
 &=  \frac{2}{Vol(\Ss^d)}\int_{\theta \in [0,\pi/2]}  Vol(\Ss^{d-1})\sin\theta(\cos \theta)^{d-1} d\theta 
 =\frac{ Vol(\Ss^{d-1})}{Vol(\Ss^d)}\int_0^1t^{d-1}dt
 .\end{align} In the above $Vol(\Ss^d)$ denotes the surface area of $\Ss^d.$ Then, we can use the following well known relations. \\
 If $d$ is even:
 $$
 Vol(\Ss^d) =  
    \frac{(2\pi)^{d/2}}{ (d-2)!!}
,~  Vol(\Ss^{d-1}) =
    \frac{2(2\pi)^{(d-2)/2}}{(d-3)!!}.
 $$
 If $d$ is odd:
  $$
 Vol(\Ss^d) =  
    \frac{2(2\pi)^{(d-1)/2}}{(d-2)!!}
,~  Vol(\Ss^{d-1}) =
    \frac{(2\pi)^{(d-1)/2}}{(d-3)!!}.
 $$
Therefore, if $d$ is even: $\frac{Vol(\Ss^{d-1})}{Vol(\Ss^d)} = \frac{(d-2)!!}{\pi(d-3)!!} \leq \sqrt{d}.$
If $d$ is odd: $\frac{Vol(\Ss^{d-1})}{Vol(\Ss^d)} = \frac{(d-2)!!}{2(d-3)!!} \leq \sqrt{d}.$
Hence, from (\ref{eq:A1}) we get
$$ \E_{s\in\Ss^d_+}\la s, r\ra  = \frac{ Vol(\Ss^{d-1})}{Vol(\Ss^d)}\frac{1}{d} 
\leq  \frac{\sqrt{d}}{d}\leq \frac{1}{\sqrt{d}}.$$
Finally,  from (\ref{eq:A11}) and  we can conclude the statement of the property:
$$\|\nabla f_{\nu}(x) - \nabla f_{\nu}(y)\|_2\leq \frac{d}{\nu \sqrt{d}}L\|x-y\|_2 = \frac{\sqrt{d} L}{\nu}\|x-y\|_2.$$

\section{
Proof of property 3) of $f_{\nu}$  defined in Section \ref{sec:oracle}: $\|\nabla f_{\nu}(x)\| \leq L$}\label{proof:A0}
	By definition we have
		$|f_{\nu}(x) - f_{\nu}(y)| = \left|\E_{b\sim U(\mathbb{B^d})} \big(f(x + \nu b) -  f(y + \nu b)\big)\right|.$	Then, using Jensen's inequality for $|\cdot|$, we can swap $\E_b$ and the absolute value $|\cdot|$ in the above, and obtain:
		$$|f_{\nu}(x) - f_{\nu}(y)| \leq 
		\E_{b\sim U(\mathbb{B^d})}|f(x + \nu b) - f(y + \nu b)|\leq  \E_{b\sim U(\mathbb{B^d})}L\|x-y\| = L\|x-y\|.$$ 
		From the above, any directional derivative is bounded by $L$:
		$$\frac{\la\nabla f_{\nu}(x),u\ra}{\|u\|} = \lim_{t\rightarrow 0}\frac{\la\nabla f_{\nu}(x),tu\ra}{\|tu\|} = \lim_{t\rightarrow 0} \frac{ |f_{\nu}(x + tu) - f_{\nu}(x)|}{\|tu\|}\leq  L  ~~\forall x, u\in \R^d.$$
		Consequently, the norm of the gradient $\nabla f_{\nu}(x)$ is bounded by $L$:
	$$\|\nabla f_{\nu}(x)\| \leq  L  ~~\forall x\in \R^d.$$
\section{Connection of Assumption \ref{assumption:2} with Mangasarian-Fromovitz Constraint Qualification (MFCQ) }\label{proof:assumption}
We define the following assumption on the constraint functions $f^i(x)$, that implies Assumption \ref{assumption:2} and that is easier to check:
\begin{assumption}\label{assumption:mfcq}
   For any point $x\in D$ there exists a direction $s_x$, such that $\la s_x, \nabla f^i(x)\ra > 0$ for all $i$ which are $\eta$-approximately active at $x$, i.e., $-f^i(x)\leq \eta.$
\end{assumption}
If we set $\eta= 0$ and consider Assumption \ref{assumption:mfcq} only for the optimal point, it reduces exactly to  MFCQ condition.

Below, we demonstrate that Assumption \ref{assumption:mfcq} is sufficient for Assumption \ref{assumption:2} to hold, i.e., that then there exists a uniform $l>0$ such that $\|\nabla f^c_{\nu}(x)\|\geq l$ for all $x$ close to the boundary: $ -f^c_{\nu}(x)\leq \eta$. 
Recall that $f^c_{\nu}(x)$  is the $\nu$-smoothed version of $\displaystyle f^c(x) = \max_{i = 1,\ldots,m} f^i(x)$ as defined in Section \ref{sec:oracle}, that is,  $\displaystyle f^c_{\nu}(x)=\E_{b\sim U(\mathbb B^d)} f^c(x+\nu b).$ 
Note that $f^c(x)$ is non-differentiable, but we can define its gradients for all the points $x$ in which the maximizing constraint function is unique. In the rest of the points, let us define the convex hull of the gradients of the maximizing functions: 
 $$\displaystyle \hat \partial f^c(x) := \text{Conv}_{ i\in I_{\max}(x)}(\nabla f^i(x)),$$
 where $I_{\max}(x)$ is the set of indices defined by
 \begin{align}\label{eq:I_max}
 I_{\max}(x) := \arg\max_{j=1,\ldots,m} f^j(x).
 \end{align}
 In the points in which $f^c(x)$ is differentiable, $\hat \partial f^c(x) = \nabla f^c(x)$. In the convex case, $\hat \partial f^c(x)$ reduces to sub-differential $\partial f^c(x)$.  Next, using Assumption \ref{assumption:mfcq}, we can demonstrate that $ \hat \partial f^c(x)$ is bounded away from zero for the points $x$ close to boundary. 
 Indeed, from Assumption \ref{assumption:mfcq} it follows that  $$\forall x\in D ~\exists s_x\in \R^d, \bar l_x>0 : ~\la s_x, \nabla f^i(x)\ra \geq \bar l_x~~~ \forall i\in I_{active}(x),$$ where $ I_{active}(x):= \{i \in \{1,\ldots,m\}:f^i(x)\geq -\eta\}$ is the set of indices of $\eta$-approximately active constraints at point $x$. 
 For $x$ close to the boundary such that $I_{active}(x) \neq \emptyset$, we have $I_{\max}(x) \subseteq I_{active}(x),$ where $I_{\max}(x)$ determines $\hat \partial f^c(x)$ (\ref{eq:I_max}).  Then, any convex combination of the gradients of $\eta$-approximately active constraints $\nabla f^i(x), i\in I_{\max}(x)$, is bounded away from zero: 
    $$\left\la \sum_{ i\in I_{\max}(x)} \beta_i \nabla  f^{i}(x), s_x\right\ra > \bar l_x,~ \forall \beta_i:\sum_{i\in I_{\max}(x)} \beta_i = 1, \beta_i \geq 0,$$

    That means $0\notin \hat \partial f^c(x).$ Hence, if $\hat \partial f^c(x)$ is sufficiently  bounded away from zero, then the gradient of the smoothed version $\nabla f^c_{\nu}(x)$ is also bounded away from zero, and Assumption \ref{assumption:2} holds. 

\section{Proof of lower confidence bound $\hat\alpha_k$ defined in Section \ref{sec:barrier_estimator}
}\label{proof:A}
We introduce the following notation for all the proofs below:
\begin{align}\label{def:barde}
\bar \delta := \frac{\delta}{2K+1}.
\end{align}
Using Hoeffding's inequality, for zero-mean independent $\sigma^2$-sub-Gaussian random variables $X_i$, $i = 1,\ldots,n$, for any $\lambda >0$ we have the bound
$$\Prob\left\{\left|\sum_{i = 1}^n X_i\right|  \geq \lambda \right\}\leq e^{-\frac{\lambda^2}{n\sigma^2}}.$$
Setting the probability in the previous expression to $\bar\de = e^{-\frac{\lambda^2}{n\sigma^2}}$, we get, $ \lambda = \sqrt{n} \sigma\sqrt{\ln \frac{1}{\bar\de}}$, and thus, we derive the following confidence bounds:
\begin{align}\label{eq:hoeffding}
\Prob\left\{\left|\frac{\sum_{i = 1}^n X_i}{n} \right| \leq \frac{\sigma}{\sqrt{n}}\sqrt{\ln\frac{1}{\bar\de}}\right\}\geq 1-\bar\de.\end{align}
Second, note that $F^i(x_k, \xi^{i-}_{kl}) - f^i(x_k) = \xi^{i-}_{kl}$ for $l = 1,\ldots,n_k, i = 0,\ldots,m$ are $\sigma^2$-sub-Gaussian independent random variables.  Then, from (\ref{eq:hoeffding}) for any $\bar\de \in (0,1)$ we have $\forall i=0,\ldots,m$ 
$$\Prob\left\{ \left|\sum_{l = 1}^{n_k}\frac{ F^i(x_k,\xi^{i-}_{kl}) - f^i(x_k)}{n_k}\right| \leq \frac{\sigma \sqrt{\ln\frac{1}{\bar\de}}}{\sqrt{n_k}}\right\}\geq 1-\bar\de.$$  
Thus,
$\hat F^i(x_k) : = \frac{\sum_{l = 1}^{n_k} F^i(x_k,\xi^{i-}_{kl})}{n_k} + \frac{\sigma}{\sqrt{n_k}} \sqrt{\ln \frac{1}{\bar\de}}$
indeed expresses an upper confidence bound on $f^i(x_k)$: 
\begin{align}
\Prob\{ f^i(x_k)\leq \hat F^i(x_k)\}\geq 1-\bar\de.
\end{align}
Third, we demonstrate that $\displaystyle\max_{i = 1,\ldots,m} \hat F^i(x_k)$ expresses an upper confidence bound on $\displaystyle f^c(x_k) = \max_{i = 1,\ldots,m} f^i(x_k)$: 
\begin{align}\label{eq:fc}
\Prob \left\{f^c(x_k)  \leq \max_{i = 1,\ldots,m} \hat F^i(x_k)\right\}\geq 1-\bar\de.
\end{align}
Indeed,  the above condition (\ref{eq:fc}) can be violated \textit{only if}  the upper bound $f^{j}(x_k) \geq \hat F^{j}(x_k)$ is violated for all $\displaystyle  j \in I_{\max}(x_k)$, 
with $\displaystyle I_{\max}(x_k) = \arg\max_{j=1,\ldots,m} f^j(x_k)$ as defined in (\ref{eq:I_max}). The probability of such violation event is indeed bounded by $\bar \de$:
$$\Prob\{\forall j \in I_{\max}: \hat F^j(x_k) \leq f^j(x)\} = \Prob\{\cap_{j \in I_{\max}} \hat F^j(x_k) \leq f^j(x)\} \leq \min_{j \in I_{\max}} \Prob\{ \hat F^j(x_k) \leq f^j(x)\} \leq  \bar\de.$$
Denote $\displaystyle \bar F^c(x_k) :=  \max_{i=1,\ldots,m}\frac{\sum_{l = 1}^{n_k} F^i(x_k,\xi^{i-}_{kl})}{n_k}.$ Then, $ \displaystyle \bar F^c(x_k) + \frac{\sigma}{\sqrt{n_k}} \sqrt{\ln \frac{1}{\bar\de}} =\max_{i = 1,\ldots,m} \hat F^i(x_k)$  is an upper bound on $f^{c}(x_k)$ (\ref{eq:fc}).  
Finally, note that  the properties of \textit{smoothed} functions defined in Section \ref{sec:oracle} (proved in \citep{flaxman2005online,hazan2016graduated}) imply
$|f^c_{\nu}(x_k) - f^c(x_k)| \leq \nu L.$ Consequently, we can bound $\max\{f^c_{\nu}(x_k),f^c(x_k)\}  \leq f^c(x_k) + \nu L.$ Then, bound (\ref{eq:fc}) combined with the above directly implies the bound: 
\begin{align}\label{eq:1}
\Prob \left\{
\max\{f^c_{\nu}(x_k),f^c(x_k)\}  \leq 
\hat F^c_{\nu}(x_k)
\right\} \leq 1-\bar \de.\end{align}
For the absolute values $ |f^c_{\nu}(x_k)|, | f^c(x_k)|$ we have:
\begin{align}\label{eq:2}
 \max\{f^c_{\nu}(x_k),f^c(x_k)\} = -\min\{- f^c_{\nu}(x_k), - f^c(x_k)\} \leq -\min\{|f^c_{\nu}(x_k)|, | f^c(x_k)|\}. 
\end{align}
From the definition of $\hat\alpha_k$ it follows that:
\begin{align}\label{eq:3}
-\hat \alpha_k = -| \hat F^c_{\nu}(x_k) | \leq  -\hat F^c_{\nu}(x_k);
\end{align}
Then, combining the inequalities (\ref{eq:1}-\ref{eq:3}) together, we obtain that $\hat \alpha_k$ indeed expresses a probabilistic lower bound on both $|f^c_{\nu}(x_k)|$ and $ |f^c(x_k)|$:
$$\Prob\left\{\hat \alpha_k\leq \min\{|f^c_{\nu}(x_k)|, |f^c(x_k)|\}\right\}\geq 1-\bar\de.$$
In Figure \ref{fig:illustration} below, we graphically illustrate the relations between $f^c(x_k), f^c_{\nu}(x_k)$, and $-\hat \alpha_k = \hat F^c_{\nu}(x_k)$, which can help to understand the proofs above.
\begin{figure}[H]\centering
\includegraphics[width =0.9\textwidth]{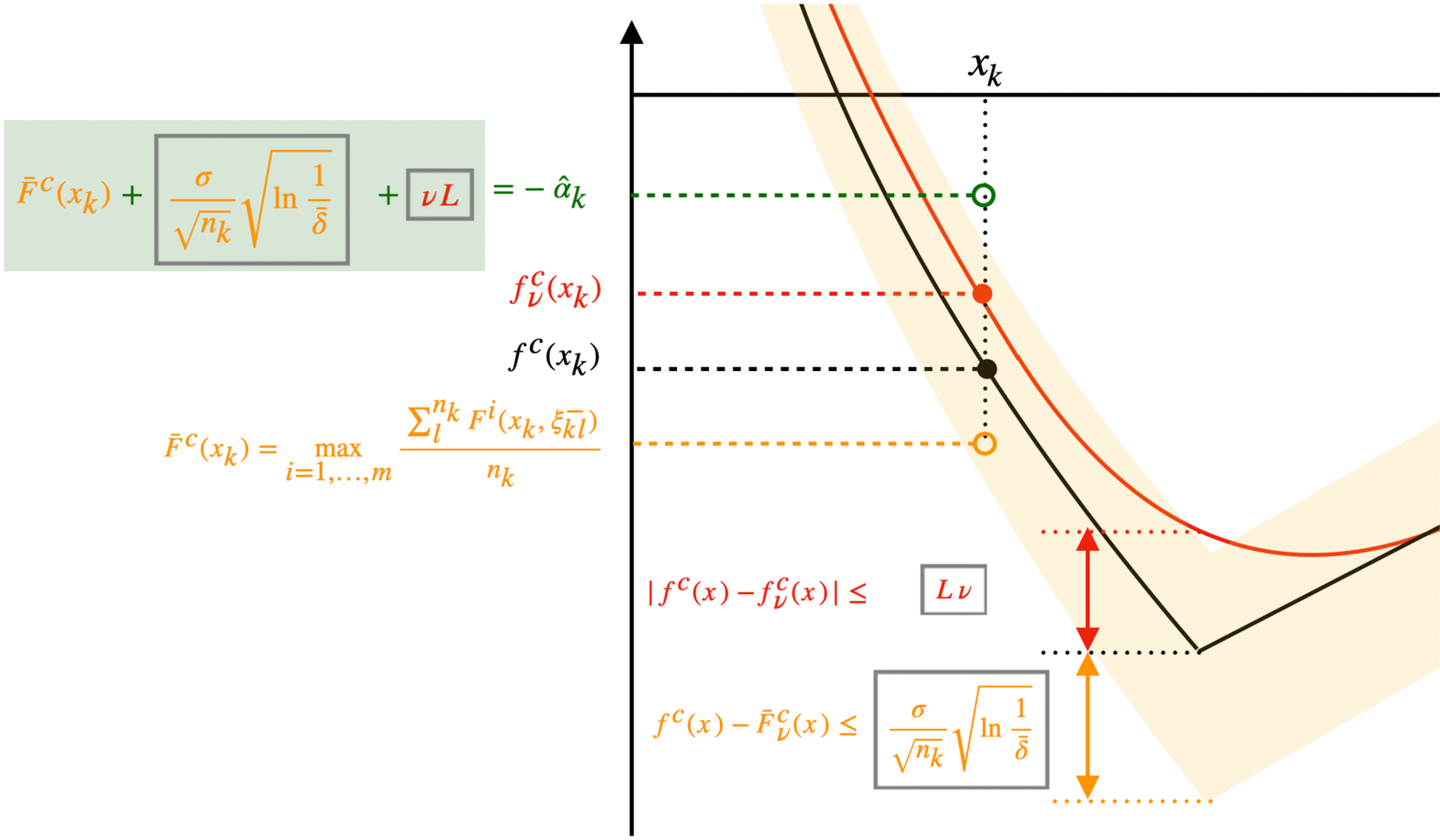}\caption{Illustration of constraint function $f^c(x_k)$, smoothed constraint function $f^c_{\nu}(x_k)$, upper bound $-\hat\alpha_k$. The dots with empty interior represent random values $-\hat\alpha_k$, $\bar F^c(x_k)$, the dots with filled interior represent deterministic values $f^c(x_k)$, $f^c_{\nu}(x_k)$.  The orange area represents the set in which $\bar F^c(x)$ lie with high probability.}
\label{fig:illustration}
\end{figure}

\section{
Proof of Fact \ref{lemma:0}}\label{proof:B} 
The  deviation of the gradient estimators $\Delta^j_k = G^j(x_k,\nu) - \nabla f^j_{\nu}(x_k)$, by definition can be expressed as follows for $j = 0,c$
\begin{align}\label{eq:delta_0}
\Delta^j_k
= \frac{1}{n_k}\sum_{l=1}^{n_k}
\left[\underbrace{\left(d\frac{f^j(x_{k}+\nu s_{kl}) - f^j(x_k)}{\nu}s_{kl} - \nabla f^j_{\nu}(x_k)\right)}_{v_l^j} +  \underbrace{d\frac{\xi_{kl}^{j+} - \xi_{kl}^{j-} }{\nu}s_{kl}}_{u_l^j}\right],
\end{align}
where the first term under the summation $v_l^j$ is dependent only on random $s_{kl}$, however the second term is dependent on both random variables coming from the noise $\xi^{j\pm}_{kl}$ and from the direction $s_{kl}$.  In the above 
\begin{align}
\label{eq:xi1}
\xi^{0-}_{kl} &:= F^0(x_k, \xi^{0-}_{kl}) - f^0(x_k)\\
\label{eq:xi2}
\xi^{0+}_{kl} &:= F^0(x_k +\nu  s_{kl}, \xi^{0+}_{kl}) - f^0(x_k +\nu  s_{kl})\\
\label{eq:xi3}
\xi^{c-}_{kl} &:= \max_{i = 1,\ldots,m} F^i(x_k, \xi^{i-}_{kl}) - \max_{i = 1,\ldots,m} f^i(x_k)\\
\label{eq:xi4}
\xi^{c+}_{kl} &:= \max_{i = 1,\ldots,m} F^i(x_k +\nu  s_{kl},  \xi^{i+}_{kl}) - \max_{i = 1,\ldots,m} f^i(x_k +\nu  s_{kl})
\end{align} 
 First, we are going to bound the variance $\E\|\Delta_{k}^j\|^2$. Using (\ref{eq:delta_0}), we obtain
\begin{align}\label{eq:delta_var}
n_k^2\E\|\Delta_k^j\|^2 &= \E\|\sum_{l=1}^{n_k} v_l^j\|^2 + \E\|\sum_{l=1}^{n_k} u_l^j\|^2 + 2\E \la\sum_{l=1}^{n_k}  v_l^j, \sum_{l=1}^{n_k} u_l^j\ra \\
&= \E\|\sum_{l=1}^{n_k} v_l^j\|^2 + \E\|\sum_{l=1}^{n_k} u_l^j\|^2 + 2\E\sum_{l=1}^{n_k}  \la v_l^j,  u_l^j\ra + 2\E\sum_{m\neq l}^{n_k,n_k} \la v_m^j,  u_l^j\ra.
\end{align}
Next, we are going to bound each of the terms in the summand above. 
\paragraph{$\E\|\sum_{l=1}^{n_k} v_l^j\|^2$:} From Stokes' theorem \citep{flaxman2005online} we know that $\E v_l^j = 0$. 
Using $L$-Lipschitzness of $f^j(x)$ for both $j = 1,c$ we can derive the following bound on its norm $\|v_l^j\|$ : 
\begin{align}\label{eq:||v_l^j||}
\|v_l^j\| = \left\|d\frac{f^j(x_k+\nu s_{kl}) - f^j(x_k)}{\nu} s_{kl} - \nabla f^{j}_{\nu}(x_k)\right\| \leq (d+1)L.
\end{align} 
Hence, $v_{l}^j$ is centered bounded random vector, and $\|v_l^j\|\leq (d+1)L$ holds with probability 1. 
Since $\{v_l^j\}_{l=1,\ldots,n_k}$ are i.i.d. zero-mean  variables, we have 
\begin{align}\label{eq:E||v_l^j||}
\E\left\|\sum_{l=1}^{n_k} v_l^j\right\|^2 = \sum_{l=1}^{n_k}\E\|v_l^j\|^2 = n_k \E\|v_l^j\| \leq n_k(d+1)^2L^2
\end{align} 
Thus, random variable $\sum_{l=1}^{n_k} v_l^j$ is bounded and zero-mean with $\E\|\sum_{l=1}^{n_k} v_l^j\|^2 \leq n_k(d+1)^2L^2$. 

\paragraph{$\E\|\sum_{l=1}^{n_k} u_l^j\|^2$:}
First, we build the bounds on $\xi^{0\pm }_{kl}$ and $\xi^{c \pm}_{kl}$ (\ref{eq:xi1}-\ref{eq:xi4}): \begin{itemize}
\item[1)] Recall that  $\xi^{0\pm}_{kl}$(\ref{eq:xi1},\ref{eq:xi2}) are zero-mean $\sigma^2$-sub-Gaussian independent random variables. and hence, from (\ref{eq:hoeffding}) it follows that
\begin{align}\E \xi^{0\pm}_{kl}=0,~\E(\xi^{0\pm}_{kl})^2\leq \sigma^2,
~\Prob\left\{\left|\frac{1}{n_k}\sum_{l= 1}^{n_k}\xi^{0\pm}_{kl}\right| \leq \frac{\sigma}{\sqrt{n_k}} \sqrt{\frac{1}{\bar \de}} \right\}\geq 1-\bar\de . 
\end{align}
\item[2)]For $\xi^{c\pm}_{kl} $(\ref{eq:xi3}, \ref{eq:xi4}) we have
 $\displaystyle |\xi^{c\pm}_{kl}| \leq \max_{i = 1,\ldots,m}|\xi^{i\pm}_{kl}|.$ 
 Indeed, recall that $$\xi^{c -}_{kl} = \max_{i = 1,\ldots,m} (f^i(x_k) +  \xi^{i-}_{kl}) - \max_{i = 1,\ldots,m} f^i(x_k).$$ Note that $$\displaystyle \max_{i = 1,\ldots,m} f^i(x_k) -  \max_{i = 1,\ldots,m} |\xi^{i-}_{kl}|\leq \max_{i = 1,\ldots,m} (f^i(x_k) +  \xi^{i-}_{kl})\leq \max_{i = 1,\ldots,m} f^i(x_k) +  \max_{i = 1,\ldots,m} |\xi^{i-}_{kl}|.$$ Hence,  $\displaystyle |\xi^{c-}_{kl}| =  \left|\max_{i = 1,\ldots,m} (f^i(x_k) +  \xi^{i-}_{kl}) - \max_{i = 1,\ldots,m} f^i(x_k)\right|\leq \max_{i = 1,\ldots,m} |\xi^{i-}_{kl}|.$ The same holds for $\xi^{c+}_{kl}.$ Consequently $\E|\xi^{c\pm}_{kl}|^2 \leq \sigma^2.$ 
\end{itemize}
From the above, using the independence of $\xi^{j\pm}_{kl}$ and $s_{kl}$ we derive
\begin{align}\label{eq:E u_l^j}
\E\sum_{l = 1}^{n_k} u_l^j = \frac{d}{\nu} \E\left(\sum_{l = 1}^{n_k} (\xi^{j+}_{kl}  -  \xi^{j-}_{kl}) s_{kl}\right) = 0.\end{align}
Also, using $\|s_{kl}\| = 1$ we obtain
\begin{align}\label{eq:E||u_l^j||}
\E \left\|\sum_{l = 1}^{n_k} u_l^j\right\|^2 = \E \frac{d^2}{\nu^2}\left\|\sum_{l = 1}^{n_k} (\xi^{j+}_{kl} - \xi^{j-}_{kl} ) s_{kl}\right\|^2 \leq n_k \frac{d^2\sigma^2}{\nu^2}.\end{align}

 \paragraph{ $\displaystyle 2\E \left\la\sum_{l=1}^{n_k}  v_l^j, \sum_{l=1}^{n_k} u_l^j\right\ra$: } 

This expression can be split into the following terms $$2\E \left\la\sum_{l=1}^{n_k}  v_l^j, \sum_{l=1}^{n_k} u_l^j\right\ra =  \E \left(2 \sum_{l=1}^{n_k}  \la v_l^j, u_l^j\ra + 2\sum_{m\neq l}^{n_k} \la v_m^j,  u_l^j\ra\right) = 2\sum_{l=1}^{n_k} \E \la v_l^j,  u_l^j\ra + 2\sum_{m\neq l}^{n_k} \E \la v_m^j,  u_l^j\ra.$$
For $m \neq l$ we have that $v_m^j$ and $u_l^j$ are independent, thus 
$$2\sum_{l=1}^{n_k}  \E\la v_l^j,  u_l^j\ra + 2\sum_{m\neq l}^{n_k} \la \E v_m^j,  \E u_l^j\ra = 2\sum^{n_k}_{l=1}  \E\la v_l^j,  u_l^j\ra.$$
Here each summand is bounded by
\begin{align*} \E\la v_l^j,  u_l^j\ra &=2\frac{d^2}{\nu^2} \E(\xi^{j+}_{kl} - \xi^{j-}_{kl}) (f^j(x_{k}+\nu s_{kl}) - f^j(x_k))\la s_{kl}, s_{kl}\ra - \la \nabla f^j_{\nu}(x_k), d\frac{\xi_{kl}^{j+} - \xi_{kl}^{j-}}{\nu} s_{kl} \ra\\ 
&= 
2\frac{d^2}{\nu^2} \E(\xi^{j+}_{kl} - \xi^{-}_{kl}) (f^j(x_{k}+\nu s_{kl}) - f^j(x_k)) \leq 2\frac{d^2}{\nu^2} \E|\xi^{j+}_{kl} - \xi^{j-}_{kl}|\E |f^j(x_{k}+\nu s_{kl}) - f^j(x_k)| \end{align*}
For the case $j = 0$ we have $\E\la v_l^0,  u_l^0\ra \leq 0.$ For the case $j = c$ we have $\E\la v_l^c,  u_l^c\ra \leq 4\frac{d^2}{\nu^2}\sigma L \nu.$
\paragraph{$\|\Delta_k^j\|$:}
Combining the bounds obtained above with (\ref{eq:delta_var}), we get the final bound on $\E\|\Delta_k^j\|^2:$
\begin{align}\label{eq:expectation_de}
\E \|\Delta^j_k\|^2  
&\leq \frac{1}{n_k}\left((d+1)^2L^2 + \frac{2 d^2 \sigma^2}{\nu^2} + \frac{4 d^2 \sigma L}{\nu}
\right) \leq \frac{(d+1)^2}{n_k}\left(L + \frac{\sqrt{2}\sigma}{\nu}\right)^2 
.\end{align} 
Note that $\E \Delta_{k}^j = \frac{1}{n_k}\sum_{l=1}^{n_k} (\E v_l^j + \E u_l^j) = 0$ and that $\Delta_{k}^j$ is sub-Gaussian random vector due to the fact that $v_l^j$ and $u_l^j$ are zero-mean sub-Gaussian. From property  (\ref{eq:hoeffding}) of sub-Gaussian variables 
we have
\begin{align}
\displaystyle \Prob\left\{\|\Delta^j_k\|\leq   
\frac{d+1}{\sqrt{n_k}}\left(L + \frac{\sqrt{2}\sigma}{\nu}\right)\sqrt{\ln\frac{1}{\bar\de}}\right\} 
\geq 1-\bar\de.
\end{align}
Substituting our notation of $\Sigma: = (d+1)\sqrt{\ln\frac{1}{\de} + \ln (2K+1)}\left(\sqrt{2}\sigma + L \nu\right)$ defined in Section \ref{sec:safety}, we get
\begin{align}\label{eq:deltak}
\displaystyle \Prob\left\{\|\Delta^j_k\|\leq  \frac{\Sigma}{\nu\sqrt{n_k}}\right\} 
\geq 1-\bar\de.
\end{align}
 Recall $\bar \delta = \frac{\delta}{2K+1}\leq \frac{\delta}{K} .$ Using the probability of union of the above events $A_k$ defined in (\ref{eq:deltak}) inequality 
 \begin{align}\label{eq:union}
 \displaystyle 1-\Prob\left\{\cup_{k = 1}^{K} A_k\right\} \geq 1-\sum_{k = 1}^K \Prob\{A_k\},
 \end{align}
 we get the result holding for $K$ points with probability $1-K\bar\de \geq 1-\de$.

\section{
Proof of Fact \ref{lemma:2}}\label{proof:C}
Here we derive a bound on the norm of the deviation $\zeta_k = g_k - \nabla B_{\eta,\nu}(x_k)$. By definition, we have:
\begin{align}\label{eq:zeta}
&\|\zeta_k\| = \left\|g_k - \nabla B_{\eta,\nu}(x_k)\right\| =  \left\|G^0(x_k, \nu,\xi_k^0,s_k) - \nabla f^0_{\nu}(x_k)  + \eta\left(\frac{G^c(x_k, \nu, s_k, \xi_k)}{|\hat F^c_{\nu}(x_k)|} - \frac{\nabla f^c_{\nu}(x_k)}{|\hat f^c_{\nu}(x_k)|} \right)\right\|\\
&
\leq \|\Delta^0_k\| + \eta\left\|\frac{G^c(x_k, \nu, s_k, \xi_k)}{\hat \alpha_k} \pm \frac{\nabla f^c_{\nu}(x_k)}{\hat\alpha_k} - \frac{\nabla f^c_{\nu}(x_k)}{\alpha_k} \right\|
\leq \|\Delta_k^0\| + \eta \frac{\|\Delta_k^c\|}{\alpha_k} + \eta \frac{\|\nabla f^c_{\nu}(x_k)\|(\hat \alpha_k - \alpha_k) }{\hat \alpha_k^2}.\nonumber
\end{align}
 Recall the notation $\Sigma:= (d+1)\sqrt{\ln\frac{1}{\bar\de} + \ln K}\left(\sigma + L \nu\right).$ 
 Next, we combine the result of Fact \ref{lemma:0} with the above bound (\ref{eq:zeta}), and use the choice $\nu= \frac{C\eta}{L}.$
We obtain that with probability greater than $1-\bar\de$ the following holds:
\begin{align*}
\|\zeta_k\| &
\leq \frac{\Sigma}{\nu\sqrt{n_k}}\left(1+\frac{\eta}{\alpha_k}\right) + \frac{\eta L}{\hat\alpha_k^2}\frac{\sigma\sqrt{\ln\frac{1}{\de}}}{\sqrt{n_k}}\frac{\nu}{\nu}
\leq \frac{1}{\nu\sqrt{n_k}}\left(\Sigma+\frac{\Sigma \eta}{\alpha_k} + \frac{\sigma\sqrt{\ln\frac{1}{\de}} C\eta^2}{\hat\alpha_k}\right) \leq \frac{\Sigma}{\nu\sqrt{n_k}}\left(1+\frac{ \eta}{\hat\alpha_k} + \frac{C\eta^2}{\hat\alpha_k^2}\right).
\end{align*} 
In expectation, respectively, using (\ref{eq:expectation_de}) we get
\begin{align*}
\E\|\zeta_k\| &
\leq \frac{(d+1)}{\sqrt{n_k}}\left(L + \frac{\sqrt{2}\sigma}{\nu}\right)\left(1+\frac{\eta}{\alpha_k}\right) + \frac{\eta L}{\hat\alpha_k^2}\frac{\sigma}{\sqrt{n_k}} \leq \frac{(d+1)(\sqrt{2}\sigma+L\nu)}{\nu\sqrt{n_k}}\left(1+\frac{ \eta}{\hat\alpha_k} + \frac{C\eta^2}{\hat\alpha_k^2}\right).
\end{align*}
For the cases when $\hat \alpha_k\geq C\eta$, and using the same union probability inequality as in (\ref{eq:union}), we obtain the bounds:
\begin{align*}
   &\Prob\left\{
   \|\zeta_k\|  \leq \frac{\Sigma}{\nu\sqrt{n_k}}
   \left(
  1+\frac{ 2\eta}{\hat\alpha_k} 
  \right) 
~ \forall k\leq K 
   \right\}\geq 1-\bar\de.\\
   &\E\|\zeta_k\| \leq  \frac{(d+1)(\sqrt{2}\sigma+L\nu)}{\nu\sqrt{n_k}}\left(1+\frac{ 2\eta}{\hat\alpha_k} \right).
\end{align*}

\section{
Proof of Lemma \ref{lemma:3}}\label{proof:D}
Recall that $\de = (2K+1)\bar\de$ from our notation (\ref{def:barde}). Let us denote by $A^{\alpha}_{k}$ the event $[\hat \alpha_k\geq C\eta]$ and by $A^{c}_{k}$ the event $[-f^c_{\nu}(x_k)\geq 2C\eta]$. Our goal is to show that $\alpha_k$ as well as $f^c_{\nu}(x_k)$ satisfy the above bounds  for all $k\leq K$ with high probability: $\Prob\{\cap_{k=1}^K A^c_k \cap A^{\alpha}_k\} \geq 1-\de.$ The plan of the proof is as follows. First, we show that for $n_k  \geq \frac{4\Sigma^2 (C+1)^2}{\nu^2C^2L^2}$ and for any $x_k$ we have the relation: 
\begin{align}\label{base}
     \Prob\left\{A^{\alpha}_{k}|A^{c}_{k}\right\}\geq 1-\bar\de,
\end{align} 
In particular, for iteration $k = 0$ the above (\ref{base}) combined with the condition of the lemma $\Prob\{A^c_0\} =1$ directly implies:
\begin{align}
    \Prob\{A^c_0\cap A^{\alpha}_0\} = \Prob\{A^{\alpha}_0|A^c_0\}\Prob\{A^c_0\}= 1-\bar\de.
\end{align}
Next, we are going to show that with high probability the conditions satisfaction $A^\alpha_k, A^c_k$ at iteration $k$ imply the similar events at iteration $k+1$ \begin{align}\label{induction}
\Prob\left\{A^{\alpha}_{k+1} \cap A^{c}_{k+1}| A^{\alpha}_{k} \cap A^{c}_{k}\right\} \geq (1-\bar\de)^2.\end{align}
And finally, since each iteration $x_{k+1}$ is dependent only on the observations of the previous iteration $x_{k}$,  (\ref{induction}) will lead to:
 $$\Prob\{\cap_{k=1}^K [A^{\alpha}_{k} \cap A^{c}_{k}]\} = 
 \prod_{k=0}^{K-1}\Prob\{A^{\alpha}_{k+1} \cap A^{c}_{k+1}| A^{\alpha}_{k} \cap A^{c}_{k}\}\Prob\{A^{\alpha}_0 \cap A^{c}_0\} \geq (1-\bar\de)^{2K+1} \geq 1-(2K+1)\bar\de,$$ where the last inequality holds due to Bernoulli's inequality. This is exactly the statement of the lemma. It is left to prove the statements  (\ref{base}) and  (\ref{induction}). 
 \paragraph{Proof of (\ref{base}).}
 By definition, (\ref{base}) is equivalent to
$$\Prob\left\{\hat \alpha_k \geq C\eta|-f^c_{\nu}(x_k) \geq  2C\eta \right\}\geq 1-\bar\de.$$
 Recall that 
$\nu = \frac{C\eta}{L} $, $\sqrt{n_k}\geq \frac{C+1}{C} \frac{2 \Sigma }{\nu L}$, from which it follows that  $\frac{\Sigma}{\sqrt{n_k}}\leq \frac{ \nu L }{2}.$ 
Then, from the definition of $\hat\alpha_k$, we indeed have \begin{align}\label{eq:alpha}
    \Prob\left\{
    -f^c_{\nu}(x_k)\geq \hat\alpha_k \geq -f^c_{\nu}(x_k) - \frac{\Sigma}{\sqrt n_k} - \frac{\nu L}{2}\right\} \geq 1-\bar\de ~~\Rightarrow ~~\Prob\left\{\hat\alpha_k \geq  C\eta |  -f^c_{\nu}(x_k) \geq 2C\eta\right\}\geq 1-\bar\de.
\end{align}
 
 \paragraph{Proof of (\ref{induction}).} Assuming that both $A^{\alpha}_k$ and $A^c_k$ hold, we are going to prove that $\Prob\{A^{\alpha}_{k+1}\cap A^{c}_{k+1} \}\geq 1-\bar\de.$ In particular, assume 
 \begin{align}\label{ass}
 \hat\alpha_k\geq \eta C\text{ ~and }-f^c_{\nu}(x_k) \geq 2\eta C.
 \end{align}
We split the above condition (\ref{ass}) into two cases:  $-f^c_{\nu}(x_{k})\geq 4C\eta$ and $4C\eta\geq -f^c_{\nu}(x_{k})\geq 2C\eta$.\\
\begin{itemize}
\item[\textbf{Case 1:}]  Consider $-f^c_{\nu}(x_k) \geq 4C\eta$. After one step with the step size $\gamma_k$ the value $-f^c_{\nu}(x)$ cannot decrease more than twice: 
\begin{align}
\label{}
-f^c_{\nu}(x_{k+1}) 
\geq -f^c_{\nu}(x_k)/2 \geq 2C\eta.
\end{align} This is due to the chosen step size such that $\gamma_k\leq \frac{\alpha_k}{2L\|g_k\|}$ and $L$-Lipschitz continuity of $f^c_{\nu}(x).$ In particular, from Lipschitz continuity it follows that 
\begin{align*}
f^c_{\nu}(x_{k+1}) &\leq f^c_{\nu}(x_k) - \gamma_k\la g_k,\nabla f^c_{\nu}(x_k) \ra \leq  f^c_{\nu}(x_k) + \frac{- f^c_{\nu}(x_k)}{2L\|g_k\|} L\| g_k\|  \leq \frac{f^c_{\nu}(x_k)}{2}. \end{align*}
Thus, indeed, for the first case we get
\begin{align*}
-f^c_{\nu}(x_{k+1}) &\geq \frac{-\hat f(x_k)}{2} \geq 2C\eta.
\end{align*} 
\item[\textbf{Case 2:}] Now, consider that   $2C\eta \leq -f^c_{\nu}(x_k)\leq 4C\eta.$  Hence, 
\begin{align}
\label{eq:al2}
C\eta\leq\hat\alpha_k \leq -f^c_{\nu}(x_k)\leq 4C\eta.
\end{align}
The above implies that $ -f^c_{\nu}(x_k) \leq 4C\eta \leq \eta$, since $4C = \frac{4l^2}{8 L^2} \leq \frac{1}{2} < 1.$ Then, by Assumption 2 and Property (3) of the smoothed function,  $l\leq\|\nabla f^c_{\nu}(x_k)\|\leq L.$ Thus, the step direction $g_k = \nabla B_{\eta,\nu}(x_k) + \zeta_k$ is a descent  direction, namely,  
$\la -g_k, -\nabla f^c_{\nu}(x_k) \ra \geq 0.$ In particular, it can be bounded by: 
\begin{align*}
&\la g_k, \nabla f^c_{\nu}(x_k) \ra = 
\la\nabla B_{\eta,\nu} +  \zeta_k, \nabla f^c_{\nu}(x_k) \ra   = 
\la\nabla f^0(x_k) +\eta\frac{\nabla f^c_{\nu}(x_t)}{\hat\alpha_k} + \zeta_k, \nabla f^c_{\nu}(x_k)\ra \\
&\geq  \frac{\eta}{ \hat \alpha_k}\|\nabla f^c_{\nu}(x_k)\|^2 -\|\nabla f^0(x_k)\| \| \nabla f^c_{\nu}(x_k)\| -\| \zeta_k\|\|\nabla f^c_{\nu}(x_k)\|\\
& \geq  \|\nabla f^c_{\nu}(x_k)\|\left(  \frac{\eta}{ \hat \alpha_k}\|\nabla f^c_{\nu}(x_k)\| - \| \zeta_k\| - \|\nabla f^0(x_k)\|  \right).
\end{align*}
The step size $\gamma_k$ is such that the smoothed constraint $f^c_{\nu}(x)$ decreases along it. To prove that, we use $M_{\nu}$ smoothness of $f_{\nu}^c(x)$:
\begin{align*}
f^c_{\nu}(x_{k+1}) = f^c_{\nu}(x_k + \gamma_k g_k)
&\leq f^c_{\nu}(x_k) - \gamma_k \la  g_k, \nabla f^c_{\nu}(x_k) \ra + \frac{M_{\nu}}{2} \gamma_k^2\| g_k\|^2 
\end{align*} 
Hence, the descent is bounded by
\begin{align*}
f^c_{\nu}(x_{k+1}) - f^c_{\nu}(x_k) & \leq - \gamma_k L \|\nabla f^c(x_k)\| + \frac{M_{\nu}}{2} \gamma_k^2\| g_k\|^2 \leq -\gamma_k\left(  \la  g_k, \nabla f^c_{\nu}(x_k) \ra  - \frac{M_{\nu}}{2}\gamma_k \|g_k\|^2\right) .
\end{align*}
Recall that $\gamma_k = \frac{1}{\|g_k\|}
\min\left\{\frac{\hat \alpha_k}{2Lk^{2/5}},\frac{1}{k^{3/5}}\right\}$, $ M_{\nu} \leq \frac{\sqrt{d}L}{\nu}$, and $ \nu = \frac{C \eta}{L}.$ 
Hence:
\begin{align}\label{eq:fbound}
f^c_{\nu}(x_{k+1}) - f^c_{\nu}(x_k) &  \leq -\gamma_k\left(  \la  g_k, \nabla f^c_{\nu}(x_k) \ra  - \frac{\sqrt{d}L}{2\nu}\min \left\{\frac{\hat \alpha_k}{2L k^{2/5}},\frac{1}{k^{3/5}}\right\}\|g_k\|\right) \nonumber\\
&\leq -\gamma_k
\left\la 
\nabla f^c_{\nu}(x_k), \frac{\eta}{ \hat \alpha_k}
\nabla f^c_{\nu}(x_k) 
+  \zeta_k 
+ \nabla f^0_{\nu}(x_k) \right\ra
+ \gamma_k \frac{\sqrt{d}L \hat\alpha_k \|g_k\|}{4\nu L k^{2/5}} \nonumber\\
& \leq \gamma_k \left( -\frac{\eta}{ \hat \alpha_k}\|\nabla f^c_{\nu}(x_k)\|^2 + \| \zeta_k\|L + \|\nabla f^0_{\nu}(x_k)\|L  + \frac{\|g_k\|\sqrt{d}L}{4 k^{2/5}} \right)\nonumber\\
& \leq \gamma_k 
\left(-\frac{1}{C}l^2 + 2L^2  + \frac{(2+1/C)\sqrt{d}L^2}{4 k^{2/5}} \right)\nonumber
\end{align}
In the above, the second inequality is due to the fact that $\|g_k\|$ is upper bounded by
\begin{align}
&g_k = \nabla f^0_{\nu}(x_k) + \frac{\eta}{\hat \alpha_k}\nabla f^c_{\nu}(x_k)  + \zeta_k,
\end{align}
and the last inequality is due to (\ref{eq:al2}) and Assumption 2. 
Additionally, for  $n_k\geq \left(\frac{C+1}{C}\right)^2 \frac{4\Sigma^2}{\nu^2 L^2}$, using Fact \ref{lemma:2}  we get $\Prob\left\{\|\zeta_k\| \leq \frac{\Sigma}{\nu\sqrt{n_k}} (1+\frac{1}{C})\leq \frac{L}{2} \right\}\geq 1-\bar\de,$ from which the above 
follows directly.

Recall from the definition of the Lemma \ref{lemma:3} that $C:=\frac{l^2}{8L^2}.$  Then starting from $k \geq (\sqrt{d}( 2+1/C))^{2.5},$
we obtain
$$\Prob\left\{f^c_{\nu}(x_{k+1}) - f^c_{\nu}(x_k)
\leq 0\right\}\geq 1-\bar \de.$$
Thus, for the second case we also get $$\Prob\left\{ -f^c_{\nu}(x_{k+1}) \geq  -f^c_{\nu}(x_k) \geq 2C\eta\right\}\geq 1-\bar\de.$$ 
\end{itemize}
We obtained that for both cases, given $A^{\alpha}_k$ and $A^c_k$ hold, for $x_{k+1}$ generated by ZeLoBa we can guarantee 
$\Prob\{A^c_{k+1}| A^{\alpha}_{k} \cap A^{c}_{k}\} \geq 1-\bar\de.$ Finally, using (\ref{eq:alpha})  we get $$\Prob\{ A^{\alpha}_{k+1} \cap A^{c}_{k+1} | A^{\alpha}_{k} \cap A^{c}_{k} \} = \Prob\{ A^{\alpha}_{k+1}| A^{c}_{k+1}\} \Prob\{A^c_{k+1}| A^{\alpha}_{k} \cap A^{c}_{k}\}\geq (1-\bar\de)^2.$$
This concludes the proof of (\ref{induction}).

\section{
Proof of Proposition \ref{proposition}} \label{proof:E}
Recall that $n_k\geq \frac{4C^2\Sigma^2}{\nu^2(C+1)^2L^2}.$ First, we demonstrate that all the measurement points $x_k + \nu s_{kj}$ and the next iterate $x_{k+1}$ with high probability lie inside the ball centered in $x_k$ with radius $\frac{-f^c(x_k)}{L}$. Indeed, based on Lemma \ref{lemma:3} for $\nu = \frac{C\eta}{L}$ it holds that \begin{align}\label{eq:11}
\Prob\left\{\|x_k - (x_k + \nu s_{kj})\| \leq \frac{-f^c(x_k)}{L}\right\}\geq 1-\bar\de.\end{align} Similarly, from the definition of the step size $\gamma_k$ it follows that 
\begin{align}\label{eq:22}
\Prob\left\{\|x_k - x_{k+1}\| \leq \frac{-f^c(x_k)}{2L}\right\}\geq 1-\bar\de.
\end{align}
Next, we show that for any point $y \in \R^d$ satisfying $\|x_k-y\| \leq \frac{\min_i-f^i(x_k)}{L}$ we have
$\forall i = 1,\ldots,m,~  $
\begin{align}\label{eq:33}
 f^i(y) &
 \leq   f^i(x_{k}) + \la \nabla f^i(x_{k}) , y - x_k\ra \leq  f^i(x_{k}) + L \|y - x_k\|  \leq  f^i(x_{k}) + L \frac{-f^i(x_k)}{L}\leq 0.
\end{align}
Moreover, if it satisfies $\|x_k-y\| \leq \frac{\min_i-f^i(x_k)}{2L}$ we have
$\forall i = 1,\ldots,m,~  $
\begin{align*}
 f^i(y) &\leq  f^i(x_{k}) + L \frac{-f^i(x_k)}{2L}
\leq \frac{1}{2}f^i(x_{k}) \leq 0.
\end{align*}
Hence, $f^i(y) \leq \frac{1}{2}f^i(x_{t})\leq 0.$ 

Combining (\ref{eq:33}) with (\ref{eq:11}) and (\ref{eq:22}) respectively, we get
\begin{align}\label{eq:5}
\Prob\left\{f^i(x_{k} + \nu s_{kj})\leq 0|f^i(x_{k})\leq 0\right\}\geq 1-\bar\de.\end{align}
\begin{align}\label{eq:4}
\Prob\left\{f^i(x_{k+1})\leq 0|f^i(x_{k})\leq 0\right\}\geq 1-\bar\de.\end{align}
Using the fact that the process is Markovian, using (\ref{eq:4}) we derive
\begin{align}\label{eq:6}
    \Prob\left\{f^i(x_{k})\leq 0~\forall k\leq K\right\} = \prod_{k=0}^{K-1}\Prob\left\{f^i(x_{k+1})\leq 0|f^i(x_{k})\leq 0\right\}\Prob\left\{f^i(x_{0})\leq 0\right\} \geq (1-\bar\de)^K \geq 1-\de.
\end{align}
Then, from (\ref{eq:6}) combined with (\ref{eq:5}) it follows that 
\begin{align}\label{eq:7}
    &\Prob\left\{f^i(x_{k}+\nu s_{kj})\leq 0~\forall k\leq K\right\} \\
    &\geq \prod_{k=1}^K\Prob\left\{f^i(x_{k} +\nu s_{kj})\leq 0|f^i(x_{k})\leq 0\right\} \Prob\left\{f^i(x_{k})\leq 0~\forall k\leq K\right\} \\
    &\geq (1-\bar\de)^{2K} \geq 1-\de,
\end{align}
which finalizes the proof.

\section{
Proof of Theorem \ref{theorem:4}}\label{proof:F}
The plan of our proof is as follows. First, using the local smoothness we bound the minimal decrease in the barrier function value after single iteration $k$ by the expression dependent on the barrier gradient norm. Next, summing these bounds together for all $k\leq K$ and noting that the barrier value cannot decrease forever, we can derive an upper bound on the weighted sum of the barrier gradient norms $\|\nabla B_{\eta,\nu}(x_k)\|$. After, we note that the above weighted sum divided by the sum of these weights expresses the expected value for the randomly selected $x_R$ defined in Step 8 of ZeLoBa algorithm. Hence, we obtain the corresponding bound on $\E_R\|\nabla f(x_R)\|$. And finally, we show that the derived bound entails the satisfaction of the approximate KKT conditions.

\paragraph{Bounding the decrease in the barrier value using the local smoothness.} 
In the paper \citep{hinder2019poly} the authors have shown that 
$ M_2(x_k) = M\left(1+\sum _{i=1}^m\frac{2\eta}{-f^i(x_k)}\right) + \sum_{i=1}^m\frac{4 \eta L}{ (-f^i(x_k))^2} $ 
represents a \textit{local} Lipschitz constant of  $\nabla B_{\eta,\nu}(x)$ at the  point $x_k$. If we replace all $f^i(x)$ by single smooth constraint $f^c_{\nu}(x)$. Then its gradient is locally $M_{2,\nu}(x_k)$-smooth at the point $x_k$, and with probability $1-\bar\de$ we can bound: 
	\begin{align}\label{l2:bound}
	M_{2,\nu}(x_k) 
	&\leq M_{\nu}\left(1+\frac{2\eta}{\hat \alpha_k}\right) + \frac{4L^2\eta}{\hat \alpha_k^2}.
	\end{align} 
	That means that the gradient of the smoothed barrier around any point $x_k\in D$ is smooth with $M_{2,\nu}(x_k).$  Recall that $\hat \gamma_k = \min\{ \frac{\hat \alpha_k}{2L k^{2/5}}, \frac{1}{k^{3/5}}\}.$ Let $\hat \gamma_k = \gamma_k\|g_k\|.$ 
	
	Then, at each iteration of \Cref{safe-barrier-0} the value of the barrier decreases at least by the following value:
	\begin{align}\label{eq:barrier_value}
&B_{\eta,\nu}(x_k) - B_{\eta,\nu}(x_{k+1}) \geq - \gamma_k \la \nabla B_{\eta,\nu}(x_k),g_k\ra - \frac{1}{2} M_{2,\nu}(x_k) \gamma_k^2 \|g_k\|^2\nonumber\\
&=\gamma_k \la \zeta_k,g_k\ra + \gamma_k\|g_k\|^2- \frac{1}{2} M_{2,\nu}(x_k) \gamma_k^2 \|g_k\|^2\nonumber\\
    &	\geq \|g_k\| \hat \gamma_k +   \frac{\la \zeta_k, g_k\ra}{\|g_k\|}\hat \gamma_k - \left(M_{\nu} + \frac{\eta M_{\nu}}{\hat\alpha_k} + \frac{4\eta L^2}{\hat\alpha_k^2} \right)\hat\gamma_k^2\nonumber\\
	&\geq \|\nabla B_{\eta,\nu}(x_k)\| \hat \gamma_k - 2\| \zeta_k\|\hat \gamma_k - \left(M_{\nu}+\frac{\eta M_{\nu}}{\hat\alpha_k} + \frac{4\eta L^2}{\hat\alpha_k^2} \right)\hat\gamma_k^2,
    \end{align}
    	where the first inequality holds due to the Taylor's theorem and $M_{2,\nu}(x)$-local smoothness of the barrier.
\paragraph{Deriving the bound on $\E\|\nabla B_{\nu,\eta}(x_R)\|.$}
	Using (\ref{eq:barrier_value}), we can derive the following bound using telescopic sum over all $k\leq K$:
	\begin{align*}
	& B_{\eta,\nu}(x_0) - B_{\eta,\nu}(x_{\eta}^*) \geq B_{\eta,\nu}(x_0) - B_{\eta,\nu}(x_K) = \sum_{k=0}^{K} \left( B_{\eta,\nu}(x_k) - B_{\eta,\nu}(x_{k+1})\right) 
	\\
	& \geq \sum_{k=0}^{K}  \|\nabla B_{\eta,\nu}(x_k)\|\hat\gamma_k - \sum_{k=0}^{K}  2\| \zeta_k\|\hat \gamma_k - \sum_{k=0}^{K} \frac{M_{\nu}\hat\gamma_k ^2}{2} - \sum_{k=0}^{K} \frac{ M_{\nu}\eta\hat \gamma_k^2}{2\hat \alpha_k}.
	\end{align*}
 We define $D_f = B_{\eta,\nu}(x_0) - B_{\eta,\nu}(x_{\eta}^*)$.
    Then, we can bound the weighted sum of $\|\nabla B_{\eta,\nu}(x_k)\|$ as follows
	\begin{align*}
	& \sum_{k=0}^{K} \|\nabla B_{\eta,\nu}(x_k)\|\hat \gamma_k \leq
	D_f+\sum_{k=0}^{K}  2\| \zeta_k\|\hat \gamma_k + \sum_{k=0}^{K} \frac{M_{\nu}\hat\gamma_k ^2}{2}+\sum_{k=0}^{K} \frac{ M_{\nu}\eta\hat \gamma_k^2}{2\hat \alpha_k} + \sum_{k=0}^{K}  \frac{4\eta L^2  \hat \gamma_k^2}{2\hat\alpha_k^2}.
	\end{align*}
Hence, if $R$ is a random variable such that $\Prob\{ R  = k\}=  \frac{\hat\gamma_k}{\sum_{k=1}^K \hat\gamma_k }$, defined in Step 8 of ZeLoBa algorithm, then
	\begin{align*}
	& 
	\E\|\nabla B_{\eta,\nu}(x_R)\| \leq \frac{1}{\sum_{k=0}^K\hat\gamma_k}\bigg(D_f+\sum_{k=0}^{K}  2\E\|\zeta_k\|\hat \gamma_k + \sum_{k=0}^{K} \frac{M_{\nu}\hat\gamma_k ^2}{2}+\sum_{k=0}^{K} \frac{ M_{\nu}\eta\hat \gamma_k^2}{2\hat \alpha_k} + \sum_{k=0}^{K}  \frac{4\eta L^2  \hat \gamma_k^2}{2\hat\alpha_k^2}\bigg).
	\end{align*}
Therefore, 
	\begin{align*}
	 \E \|\nabla B_{\eta,\nu}(x_R)\| 
	& \leq \frac{D_f}{\sum\hat\gamma_k} 
    	+\frac{\sum_{k=0}^{K}  2\E\|\zeta_k\|\hat\gamma_k}{\sum_{k=0}^K\hat\gamma_k} 
    	+\frac{\sum_{k=1}^{K} \frac{M_{\nu}\hat\gamma_k ^2}{2}(1+\frac{\eta}{\hat \alpha_k})}{\sum_{k=0}^K\hat\gamma_k}  
	    + \frac{\sum_{k=1}^{K}  \frac{4\eta L^2  \hat \gamma_k^2}{2\hat\alpha_k^2}}{\sum_{k=0}^K\hat\gamma_k}\\
	& \leq \frac{ D_f}{\sum_{k=0}^K\hat\gamma_k} 
	    + 2\max_{k\leq K}\E\|\zeta_k\| 
	    +  \frac{\sum_{k=1}^{K}\frac{ 2L\sqrt{d}}{2\nu k^{6/5}}(1+\frac{1}{C})}{\sum_{k=0}^K\hat\gamma_k}
	    + \frac{\eta\sum_{k=1}^{K}\frac{1}{2 k^{4/5}}}{\sum_{k=0}^K\hat\gamma_k}  \\
	& \leq \frac{ D_f 
	                 + \frac{L\sqrt{d}}{\nu K^{1/5} }(1+\frac{1}{C}) 
	                 +\frac{ \eta K^{1/5}}{2}
	       }{\sum_{k=0}^K\hat\gamma_k} 
	       + 2\max_{k\leq K}\E\|\zeta_k\| .
	\end{align*}
 Note that $\hat \gamma_k =\frac{1}{k^{2/5}} \min\left\{\frac{\hat\alpha_k}{2L} ,\frac{1}{k^{1/5}}\right\} \geq \frac{1}{k^{3/5}}\frac{\hat\alpha_k}{2L}$, since both terms in $\min\left\{\frac{\hat\alpha_k}{2L} ,\frac{1}{k^{1/5}}\right\}$ are smaller than 1 (without loss of generality assume $f^c(x)\geq -1/L$ for all $x\in D$, otherwise we can scale the above inequality by the corresponding factor), consequently both of them are smaller than their product.
	Thus, 
	$\sum_{k = 1}^K \hat \gamma_k \geq \frac{ K^{2/5}\eta C}{2L}.$ 
	Then, we get $\frac{1}{\sum_{k = 1}^K \hat \gamma_k} \leq \frac{2L}{C K^{2/5}\eta}.$ Finally, we obtain the following bound in expectation:
\begin{align*}
    	 \E \|\nabla B_{\eta,\nu}(x_R)\| 
    	 & \leq \frac{
                	 2 L D_f 
                	 + \frac{L\sqrt{d}}{\nu K^{1/5} }(1+\frac{1}{C}) 
            	     +\frac{ \eta K^{1/5}\ln K}{2}
	         }{C \eta K^{2/5}} 
	        + 2\max_{k\leq K}\E\|\zeta_k\|.
	\end{align*}
If $K \geq \max\left\{\left(\frac{LD_f}{C\eta^2}\right)^{5/2}, \left(\frac{L^2\sqrt{d}(1+1/C)}{C^2\eta^3}\right)^{5/3}, \left(\frac{5\ln 1/\eta}{C\eta}\right)^5\right\} = \tilde O(\frac{d^{5/6}}{\eta^5})$,  $\nu \geq \frac{C\eta}{L}$, and $n_k  \geq  
\frac{4\Sigma^2 (1+1/C)^2L^2}{C^2\eta^4}$,
then 
\begin{align}\label{eq:Rbound}
    \E \|\nabla B_{\eta,\nu}(x_R)\| 
	\leq 5 \eta.
\end{align}
\paragraph{Deriving the $\eta$-approximate KKT conditions.}	
 Next, consider the pair $(x_R, \lambda_R) = (x_R, \frac{\eta}{\hat\alpha_R}).$ Now we check the approximate KKT conditions.\\
 $\eta$-KKT.1 follows with high probability from the safety Proposition \ref{proposition} and Lemma \ref{lemma:3}.  \begin{align}
 &\Prob\left\{\lambda_R:= \frac{\eta}{\hat\alpha_R}\right\} \geq 0\\
 &\Prob\{-f^c_{\nu}(x_R) \geq \hat \alpha_R \geq 0\}\geq 1-\bar\de.
 \end{align}
$\eta$-KKT.2 also holds due to Lemma \ref{lemma:3}: 
\begin{align*}
\frac{\eta}{\hat \alpha_R}(-f^c_{\nu}(x_R)) &\leq \eta + \eta \frac{\hat \alpha_{R} - f^c_{\nu}(x_R)}{\hat \alpha_{R}} \leq \eta + \eta \frac{\sigma\sqrt{\ln 1/\de}/\sqrt{n_R} + \nu L}{\hat\alpha_R} \leq \eta + \eta \frac{2 \nu L}{\hat\alpha_R}  \leq 3\eta. \end{align*} \\
$\eta$-KKT.3 follows from the definition of $\lambda_R$ and (\ref{eq:Rbound}): $$\E\|\nabla f^0_{\nu}(x_R) + \lambda\nabla f^c_{\nu}(x_R)\| = \E\|\nabla B_{\eta,\nu}(x_R)\| \leq 5\eta.$$
\end{document}